\documentclass[11pt]{article}
\usepackage[ruled,vlined,linesnumbered]{algorithm2e}
\usepackage{amsmath,amsfonts,amssymb,amsthm,latexsym}
\usepackage{epsfig}
\usepackage{color}
\usepackage{enumerate}
\usepackage{soul}
\usepackage{hyperref}
\usepackage{enumerate}
\usepackage{graphicx}
\usepackage[cp1250]{inputenc}
\usepackage{tikz}
\usepackage{tkz-graph}
\usepackage{tkz-berge}
\usetikzlibrary{arrows.meta,shapes.arrows}
\hypersetup{colorlinks=true, linkcolor=blue, citecolor=blue,
urlcolor=blue}
\usetikzlibrary{calc}

\hypersetup{colorlinks=true}

\hypersetup{colorlinks=true, linkcolor=blue, citecolor=blue,urlcolor=blue}

\newcommand{\roeo}{\rho_{e}^{o}}%

\def\cp{\,\square\,}

\hypersetup{colorlinks=true}

\hypersetup{colorlinks=true, linkcolor=blue, citecolor=blue,urlcolor=blue}

\title{Induced matching vs edge open packing: trees and product graphs}
\date{}
\author{{Bo\v{s}tjan Bre\v{s}ar$^{1,2}$, Tanja Dravec$^{1,2}$, Jaka Hed\v{z}et$^{1,2}$ and Babak Samadi$^{2}$}\vspace{1.25mm}\\
$^{1}$Faculty of Natural Sciences and Mathematics, University of Maribor, Slovenia\\
$^{2}$Institute of Mathematics, Physics and Mechanics, Ljubljana, Slovenia\vspace{1mm}\\
{bostjan.bresar@um.si}\\
{tanja.dravec@um.si}\\
{jaka.hedzet@imfm.si}\\
{babak.samadi@imfm.si}}
\date{}

\newenvironment{claimproof}[1]{{\it\noindent{Proof.}}\space#1}{\footnotesize \hfill \ensuremath{(\square)} \medskip}

\newtheorem{theorem}{Theorem}[section]
\newtheorem{corollary}[theorem]{Corollary}
\newtheorem{claim}{Claim}
\newtheorem{lemma}[theorem]{Lemma}
\newtheorem{observation}[theorem]{Observation}
\newtheorem{proposition}[theorem]{Proposition}
\newtheorem{obs}[theorem]{Observation}
\newtheorem{problem}{Problem}

\theoremstyle{definition}

\theoremstyle{remark}

\newcommand{\eop}{\rho_{e}^{o}}
\newcommand{\im}{\nu_{I}}

\begin{document}

\maketitle

\begin{abstract}
Given a graph $G$, the maximum size of an induced subgraph of $G$ each component of which is a star is called the edge open packing number, $\rho_{e}^{o}(G)$, of $G$. Similarly, the maximum size of an induced subgraph of $G$ each component of which is the star $K_{1,1}$ is the induced matching number, $\nu_I(G)$, of $G$. While the inequality $\rho_e^o(G)\geq \nu_{I}(G)$ clearly holds for all graphs $G$, we provide a structural characterization of those trees that attain the equality. We prove that the induced matching number of the lexicographic product $G\circ H$ of arbitrary two graphs $G$ and $H$ equals $\alpha(G)\nu_I(H)$. By similar techniques, we prove sharp lower and upper bounds on the edge open packing number of the lexicographic product of graphs, which in particular lead to NP-hardness results in triangular graphs for both invariants studied in this paper. For the direct product $G\times H$ of two graphs we provide lower bounds on $\nu_I(G\times H)$ and $\rho_{e}^{o}(G\times H)$, both of which are widely sharp. We also present sharp lower bounds for both invariants in the Cartesian and the strong product of two graphs. Finally, we consider the edge open packing number in hypercubes establishing the exact values of $\rho_e^o(Q_n)$ when $n$ is a power of $2$, and present a closed formula for the induced matching number of the rooted product of arbitrary two graphs over an arbitrary root vertex. 
\end{abstract}
\textbf{2020 Math. Subj. Class.:} 05C05, 05C69, 05C70, 05C76


\section{Introduction and preliminaries} 

Throughout this paper, we consider $G$ as a finite simple graph with vertex set $V(G)$ and edge set $E(G)$. The (\textit{open}) {\em neighborhood} of a vertex $v$ is denoted by $N_{G}(v)$, and its {\em closed neighborhood} is $N_{G}[v]=N_{G}(v)\cup \{v\}$ (we omit the index $G$ if the graph $G$ is clear from the context). The {\em minimum} and {\em maximum degrees} of $G$ are denoted by $\delta(G)$ and $\Delta(G)$, respectively. Given the subsets $A,B\subseteq V(G)$, by $[A,B]$ we denote the set of edges with one endvertex in $A$ and the other in $B$. If $X\subseteq V(G)$, then $G[X]$ denotes the subgraph of $G$ induced by the vertices of $X$. On the other hand, if $Y\subseteq E(G)$, then $G[Y]$ denotes the subgraph of $G$ induced by the endvertices of edges in $Y$.

A set $M$ of edges in a graph $G$ is a {\em matching} in $G$ if no two edges of $M$ share an endvertex. Matching theory is one of the core classical topics in graph theory; see the monograph~\cite{matching} from as early as 1986. In this paper, we are interested in matchings satisfying a stronger condition. Notably, a matching $M$ is an \textit{induced matching} in $G$ if no two edges of $M$ are joined by an edge in $G$. The \textit{induced matching number} $\im(G)$ of $G$ is the maximum cardinality of an induced matching in $G$. The problem of finding a maximum induced matching was originally introduced by Stockmeyer and Vazirani \cite{sv} (under the name ``risk-free marriage problem") and extensively investigated in literature; see, for instance, \cite{Cam,ddl,Lozin} and references therein. Note that induced matchings were also studied under the name ``strong matchings'' (see, for instance, \cite{elm,FL, gh-2005}) as well as under the name ``$2$-edge packings'' (see \cite{AS,cks}). Unlike for the standard matching problem which is solvable in polynomial time, it was proved in \cite{sv} that the decision version of the induced matching number is an NP-complete problem. In contrast, a linear-time algorithm for determining the induced matching number in trees was found in \cite{zito}.

Chelladurai et al. \cite{cks} introduced the concept of edge open packing in graphs. A subset $B\subseteq E(G)$ is an \textit{edge open packing set} (or, an {\em EOP set} for short) if for each pair of edges $e_{1},e_{2}\in B$, there is no edge from $E(G)$ joining an endvertex of $e_{1}$ to an endvertex of $e_{2}$. In other words, for any two edges $e_{1},e_{2}\in B$, 
there is no edge $e\in E(G)$ leading to the existence of a path $e_{1}ee_{2}$ or a triangle with edges $e_{1},e,e_{2}$ in $G$ (such an edge $e$ is called the \textit{common edge} of $e_{1}$ and $e_{2}$). The \textit{edge open packing number} ({\em EOP number}), $\rho_{e}^{o}(G)$, of $G$ is the maximum cardinality of an EOP set in $G$. 
The same authors also presented a real-world application of EOP sets to the packet radio networks.  An additional motivation for studying this parameter comes from the recently introduced concept of injective edge coloring of graphs \cite{bq-2018,cccd,krx-2021,msy}, since this coloring can be defined as a partition of the edge set of a graph into EOP sets. Edge open packing has also been investigated in \cite{BB} from the complexity and algorithmic points of view. 

We observe that if $B$ is an EOP set (resp. induced matching) of $G$, then $G[B]$ is a disjoint union of induced stars (resp. $K_{1,1}$-stars). In this sense, the problem of induced matching can be considered as a variation of the EOP problem. Note that an EOP set is not necessarily a matching, let alone an induced matching. 

A subset $P\subseteq V(G)$ is an \textit{open packing} (or an \textit{open packing set}) if $N(u)\cap N(v)=\emptyset$ for any distinct vertices $u,v\in P$. The \textit{open packing number}, denoted by $\rho^{o}(G$), is the maximum cardinality among all open packings in $G$ (see \cite{HS}). The names of the concepts suggest that edge open packing is the edge version of (vertex) open packing, yet one has to find a correct interpretation of this. Notably, considering the star $K_{1,r}$ with $r\geq3$, and letting $e,f\in E(K_{1,r})$ be two distinct edges, we get $N(e)\cap N(f)\neq \emptyset$,
while $E(K_{1,r})$ is an EOP set of $K_{1,r}$. Thus, the definitions of both concepts cannot be simply transferred from each other. Yet, in an open packing set, no two vertices $u$ and $v$ are connected through a third vertex. Analogously, no two edges in an EOP set are connected through a third edge, while they may or may not have a common endvertex. We also remark that an injective coloring of a graph can be defined as a partition of the vertex set into open packings~\cite{bsy-2023}, which provides another link between open packings and edge open packings. 
The concept of (vertex) open packing will turn out to be useful when dealing with bounding the EOP number of direct products of graphs in Section \ref{sec-dir}.

We refer to $|V(G)|$ and $|E(G)|$ as the {\em order} and the {\em size} of $G$, respectively. By a $\rho_{e}^{o}(G)$-set, a $\im(G)$-set, an $\alpha(G)$-set and a $\rho^{o}(G)$-set we represent an EOP set, an induced matching, an independent set and an open packing set of $G$ of cardinality $\rho_{e}^{o}(G)$, $\im(G)$, $\alpha(G)$ and $\rho^{o}(G)$, respectively, where $\alpha(G)$ stands for the independence number of $G$. For terminology and notation not explicitly defined here, we refer to \cite{West}. 

For the four standard products of graphs $G$ and $H$ (according to \cite{ImKl}), the vertex set of the product is $V(G)\times V(H)$. Their edge sets are defined as follows.
\begin{itemize}
\item In the \emph{Cartesian product} $G\square H$ two vertices are adjacent if they are adjacent in one coordinate and equal in the other.
\item In the \emph{direct product} $G\times H$ two vertices are adjacent if they are adjacent in both coordinates.
\item The edge set of the \emph{strong product} $G\boxtimes H$ is the union of $E(G\square H)$ and $E(G\times H)$.
\item Two vertices $(g,h)$ and $(g',h')$ are adjacent in the \emph{lexicographic product} $G\circ H$ if either $gg'\in E(G)$ or ``$g=g'$ and $hh'\in E(H)$''.
\end{itemize}
Each of the products $G\ast H$, where $*\in \{\square,\times,\boxtimes,\circ\}$, is associative and only the first three ones are commutative~\cite{ImKl}.
Given a vertex $g\in V(G)$, let $^{g}\! H$ denote the subgraph of $G*H$ induced by $\{(g,h)\mid h\in V(H)\}$, which we call an {\em $H$-fiber}. In a similar way we define a {\em $G$-fiber} $G^h$, where $h\in V(H)$. 

The similar nature of induced matchings and EOP sets and the inequality $\nu_{I}(G)\leq \rho_{e}^{o}(G)$, that holds for all graphs $G$, motivate us to simultaneously investigate these two concepts. Note that the problem of characterizing the graphs achieving $\nu_{I}(G)=\rho_{e}^{o}(G)$ was posed as an open problem in \cite{cks}.
Trees are a natural environment for initial investigations of graph invariants. In fact, both graph invariants that are the main theme of this paper have already been studied in trees; induced matchings in \cite{FL,zito} and EOP sets in \cite{BB}. Hence, the question in which trees these two related graph invariants are equal is a natural one. 
On the other hand, graph products are an important concept in graph theory, especially in relation with graph invariants, cf.~the monograph~\cite{ImKl} surveying product graphs. During our study of induced matchings in several graph products, we noticed that similar techniques can be applied in studying EOP sets in these products, which gave us further motivation to simultaneously investigate both concepts. Surprisingly, the two invariants (in particular, the induced matching number)  have not hitherto been investigated in graph products, except for some sporadic instances of product graphs such as hypercubes \cite{gh-2005}. 
 
The paper is organized as follows. A complete characterization of trees $T$ with the induced matching number equal to the EOP number is given in Section~\ref{sec-tree}. More formally, it is the characterization of extremal trees for the inequality $\nu_{I}(G)\leq \rho_{e}^{o}(G)$ which holds for all graphs $G$.  Roughly speaking, trees that achieve the equality are obtained from the disjoint union of $k$ spiders by adding $k-1$ edges, each having a center of a spider as an endvertex, and at least two leaves of each of the spiders keep their degree $1$ also in the resulting tree (recall that a spider is a tree obtained from a star by subdividing each of its edges exactly once). 
Next, we investigate these two parameters in all four standard product graphs. In Section \ref{sec-lex}, we prove the exact formula $\nu_{I}(G\circ H)=\alpha(G)\nu_{I}(H)$ and bound $\rho_{e}^{o}(G\circ H)$ from below and above by sharp inequalities, making use of similar techniques in all of the proofs. As an immediate corollary of these two results, we infer that the computation problems of $\nu_{I}$ and $\rho_{e}^{o}$ are NP-hard even for a special family of graphs, namely triangular graphs. In Section \ref{sec-dir}, by different approaches, we give sharp lower bounds on $\nu_{I}$ and $\rho_{e}^{o}$ of direct product graphs. It is proved in Section \ref{sec-Cp} that $\rho_{e}^{o}(G\circ H)$ is a lower bound on both $\rho_{e}^{o}(G\boxtimes H)$ and $\rho_{e}^{o}(G\square H)$. Note that this comparison, unlike in the case of (vertex) open packing, cannot be made based on the fact that both $G\boxtimes H$ and $G\square H$ are spanning subgraphs of $G\circ H$ (in fact, $\rho_{e}^{o}(G_1)$ and $\rho_{e}^{o}(G_2)$ are incomparable in general, where $G_1$ is a spanning subgraph of $G_2$). This in turn leads to lower bounds on $\rho_{e}^{o}(G\boxtimes H)$ and $\rho_{e}^{o}(G\square H)$, in terms of the EOP number and the independence number of the factors, which are widely sharp. Note that the induced matching number of hypercubes has already been determined, and thus we consider this, arguably one of the main instances of Cartesian products of graphs, also with respect to the EOP number, and among other results prove that $\roeo(Q_n)=2^{n-1}$ as soon as $n$ is a power of $2$. 
Finally, in Section~\ref{sec:rooted} we give a closed formula for $\nu_{I}(G\circ_{v}H)$, in which $G\circ_{v}H$ is the rooted product of $G$ and $H$ with root $v$. This leads to the exact formula for the induced matching number of corona product graphs. We also discuss this problem in the case of EOP.


\section{Extremal trees for the bound $\im(T)\leq \eop(T)$}\label{sec-tree}

In this section, we characterize all trees $T$ with $\im(T)=\eop(T)$. First, we recall the basic formulas about the edge open packing numbers and the induced matching numbers of paths.
\begin{proposition}\emph{(\cite{cks})}\label{Pro1} For $n\geq1$,
\[
\eop(P_n)=\begin{cases}
\frac{n+1}{2} & \textrm{if } n\equiv 3\ (\textrm{mod }4),\\
\lceil\frac{n-1}{2}\rceil & \textrm{otherwise.}
\end{cases}
\]
\end{proposition}

\begin{proposition}\label{Pro2}
For $n \geq 1$, $\im(P_n)=\lfloor \frac{n+1}{3} \rfloor.$
\end{proposition}

For a tree $T$, let $L(T)$ and $S(T)$ denote the sets of leaves and support vertices of $T$, respectively. For $\ell_1\ge 1,\ldots,\ell_n\geq1$, the {\it subdivided star} $S(\ell_1,\ldots,\ell_n)$ is the tree obtained from disjoint paths $P_{\ell_1},\ldots,P_{\ell_n}$ by joining a new vertex $x$ to exactly one leaf of each of them. The vertex $x$ is called the \textit{center} of $S(\ell_1,\ldots,\ell_n)$. Notice that $S(\ell_1)\cong P_{\ell_1+1}$, in which only one of the leaves is the center. Note that $S(2,\ldots,2)$ is called a \textit{spider} with $n$ \textit{legs}, in which $\ell_{i}=2$ for each $i\in[n]$, and is shortly denoted by $S^{n}$. Let $v$ be the center of $S^n$. We call all edges incident with $v$ the {\it internal edges}, while the other edges of $S^n$ are {\it pendant edges}. 

We say a tree $T=T(k_1,\ldots,k_n)$ is in a family $\mathcal{F}$ if it is obtained from the disjoint union of $n$ spiders $S_1=S^{k_1},S_2=S^{k_2},\ldots,S_n=S^{k_n}$, $k_i\geq 2$ for each $i\in[n]$, with centers $v_1,\ldots v_n$,  respectively, by adding $n-1$ edges $e_1,\ldots,e_{n-1}$ such that the following conditions hold:
\begin{enumerate}
\item[($C_{1}$)] for every $i\in[n-1]$, $e_i$ has an endvertex in $\{v_1,\ldots,v_n\}$,
\item[($C_{2}$)] for every $i\in[n]$, there exist at least two vertices $x,y \in L(S_i)$ with deg$_T(x)=\textrm{deg}_T(y)=1$.
\end{enumerate}

The above-mentioned edges $e_1,\ldots,e_{n-1}$ are called {\it extra edges} of $T$. Note that if $T\in {\mathcal{F}}$, then it can be recursively obtained from $T'\in \mathcal{F}$ by adding a spider $S^k$ to $T'$ and exactly one edge between $T'$ and $S^k$ so that the conditions ($C_{1}$) and ($C_{2}$)  are satisfied. Moreover, letting $T=T(k_1,\ldots,k_n)\in \mathcal{F}$, we define the tree $G_T$ as follows: $V(G_T)=\{S^{k_1},\ldots,S^{k_n}\}$ and two vertices $S^{k_i}$ and $S^{k_j}$ are adjacent in $G_T$ if and only if there is an edge in $T$ between $S^{k_i}$ and $S^{k_j}$. 

\begin{observation}\label{ob:ImEOPSpider}
If $S=S^k$ is a spider with $k\geq 2$ legs, then $\im(S)=\eop(S)=k$.
\end{observation}

\begin{observation}\label{l:structureOfEOP}
Let $S=S^k$ be a spider with $k\geq 2$ legs. Each $\eop(S)$-set $F$ is either the set of $k$ pendant edges or the set of $k$ internal edges or if $k=2$, it can also contain one internal edge and its adjacent pendant edge in $S$.
\end{observation}

\begin{lemma}\label{l:im}
If $T=T(k_1,\ldots,k_n)\in \mathcal{F}$, then $\im(T)=k_1+\ldots+k_n$. Moreover, $T$ has a unique $\im(T)$-set.
\end{lemma} 
\begin{proof}
We prove that $F=\{e\in E(T)\mid e \textrm{ is a pendant edge of } S_i, i\in [n]\}$ is a unique $\im(T)$-set. We proceed by induction on $n$. If $n=1$, then the equality $\im(T)=\im\big{(}T(k_1)\big{)}=k_1$ follows from Observation \ref{ob:ImEOPSpider}. It is also clear that the set of pendant edges of $T$ is the only induced matching of $T$ of cardinality $k_1$.

Now let $n\geq2$. Since $T=T(k_1,\ldots,k_n)\in \mathcal{F}$, $T$ is obtained from the disjoint union of $n$ spiders $S^{k_1},\ldots,S^{k_n}$ (with $k_1\geq2,\ldots,k_n\geq2$ legs, respectively) by adding $n-1$ edges $e_1,\ldots,e_{n-1}$ satisfying ($C_{1}$) and ($C_{2}$). Let $S$ be a leaf spider of $G_T$ and $T'=T-V(S)$. Without loss of generality, we may assume that $S=S^{k_n}$, which implies that $S$ is a spider with $k_n$ legs and $T'=T'(k_1,\ldots k_{n-1})\in \mathcal{F}$. Let $e$ be the only edge between $T'$ and $S$ in $T$. Let $F'$ be a $\im(T')$-set and $F_S$ a $\im(S)$-set. It follows from the induction hypothesis that $|F'|=k_1+\ldots+k_{n-1}$, $F'$ contains all pendant edges of all $n-1$ spiders of $T'$, $|F_S|=k_n$ and $F_S$ is the set of all pendant edges of $S$. Since at least one endvertex of $e$ is the center of some spider, $F'\cup F_S$ is an induced matching of $T$ of cardinality $k_1+\ldots+k_n$, which implies that $\im(T)\geq k_1+\ldots+k_n$. 

Conversely, let $F$ be a $\im(T)$-set. Suppose first that $e\in F$. Then, $|F\cap E(S)|\leq k_n-1$ as $F$ cannot contain any edge of the leg of $S$ containing the endvertex of $e$ in $S$ or if the endvertex of $e$ in $S$ is the center of $S$, then $F$ cannot contain any edge of $S$. Moreover, $F\cap E(T')$ is an induced matching of $T'$ and thus $|F\cap E(T')|\leq \im(T')=k_1+\ldots+k_{n-1}$. Hence, $|F|=|F\cap E(T')|+|F\cap E(S)|+1\leq k_1+\ldots+k_{n-1}+k_n$. If $e\notin F$, then $|F|=|F\cap E(T')|+|F\cap E(S)|\leq k_1+\ldots+k_n$, where the last inequality holds as $F\cap E(T')$ and $F\cap E(S)$ are induced matchings in $T'$ and $S$, respectively.

Hence, the set $F$ of all pendant edges of all spiders in $T$ is a $\im(T)$-set. Let $F^{*}$ be an arbitrary $\im(T)$-set. In particular, $|F^{*}|=k_1+\ldots+k_n$. If $e\notin F^{*}$, then $F^{*}=\big{(}F^{*}\cap E(T')\big{)}\cup \big{(}F^{*}\cap E(S)\big{)}$. Since $F^{*}\cap E(T')$ is an induced matching of $T'\in \mathcal{F}$, we have $|F^{*}\cap E(T')|\leq k_1+\ldots+k_{n-1}$. Moreover, $|F^{*}\cap E(S)|\leq k_n$ because $F^{*}\cap E(S)$ is an induced matching of $S$. Since $F^{*}=\big{(}F^{*}\cap E(T')\big{)}\cup \big{(}F^{*}\cap E(S)\big{)}$ and $|F^{*}|=k_1+\ldots+k_n$, it follows that $|F^{*}\cap E(T')|=k_1+\ldots+k_{n-1}$ and $|F^{*}\cap E(S)|=k_n$. Thus, $F^{*}\cap E(T')$ is a $\im(T')$-set and $F^{*}\cap E(S)$ is a $\im(S)$-set. By the induction hypothesis, both $F^{*}\cap E(T')$ and $F^{*}\cap E(S)$ are unique (consisting of pendant edges). Hence, there exists exactly one $\im(T)$-set that does not contain $e$.

Suppose now that $e\in F^{*}$. Suppose first that the center $v$ of $S$ is an endvertex of $e$. Then, $F^{*}\cap E(S)=\emptyset$. Since $F^{*}\cap E(T')$ is an induced matching of $T'$, we get $|F^{*}\cap E(T')|\leq \im(T')=k_1+\ldots+k_{n-1}$. Hence, $|F^{*}|=|F^{*}\cap E(T')|+1\leq k_1+\ldots+k_{n-1}+1<k_1+\ldots k_n$, a contradiction. Thus, $v$ is not an endvertex of $e$. Since $T\in \mathcal{F}$, ($C_{1}$) implies that the endvertex of $e$ in $T'$ is the center of a spider $S'$ in $T'$. Since $e\in F^{*}$, $F^{*}\cap E(S')=\emptyset$. Hence, $F^{*}\cap E(T')$ is an induced matching of $T'$ different from the unique $\im(T')$-set. Thus, $|F^{*}\cap E(T')|<\im(T')=k_1+\ldots+k_{n-1}$. Since $e\in F^{*}$, it is also clear that $|F^{*}\cap E(S)|\leq k_n-1$. Hence, $|F^{*}|=|F^{*}\cap E(T')|+|F^{*}\cap E(S)|+1<k_1+\ldots k_n$, a contradiction. Thus, $e$ is not contained in any $\im(T)$-set.  
\end{proof}

\begin{theorem}
\label{thm:trees1}
If $T\in {\mathcal{F}}\cup \{P_{1},P_{2}\}$, then $\im(T)=\eop(T)$.
\end{theorem}
\begin{proof}
The result is trivial when $|V(T)|\leq2$. So, we let $T=T(k_1,\ldots,k_n)\in \mathcal{F}$. We proceed by induction on $n\geq1$. If $n=1$, then it follows form Observation~\ref{ob:ImEOPSpider} that $\im(T)=\eop(T)=k_1$. Let $n\geq2$. Since an induced matching is also an EOP set, $\eop(T)\geq \im(T)=k_1+\ldots+k_n$, where the last equality follows from Lemma \ref{l:im}. Hence, it remains for us to prove that $\eop(T)\leq k_1+\ldots+k_n$.

Suppose first that there exists a $\eop(T)$-set $F$ such that $e\notin F$ for some $e\in \{e_1,\ldots,e_{n-1}\}$. Let $T_1$ and $T_2$ be the components of $T-e$. Clearly, $T_1,T_2\in \mathcal{F}$. Without loss of generality, we write $T_1=T_1(k_1,\ldots,k_\ell)$ and $T_2=T_2(k_{\ell+1},\ldots,k_n)$. Notice that $F\cap E(T_1)$ and $F\cap E(T_2)$ are EOP sets in $T_1$ and $T_2$, respectively. So, we get $|F|=|F\cap E(T_1)|+|F\cap E(T_2)|\leq \eop(T_1)+\eop(T_2)=k_1+\ldots+k_\ell+k_{\ell+1}+\ldots+k_n$ by the induction hypothesis.

Hence, we may now assume that $\{e_1,\ldots,e_{n-1}\}\subseteq F$ for any $\eop(T)$-set $F$. Let $S=S^{\ell}$ be a leaf spider of $G_T$ and $T'=T-V(S)$. It follows from the definition of $\mathcal{F}$ that $T'\in \mathcal{F}$. Moreover, we may assume that $\ell=k_n$, and hence $T'=T'(k_1,\ldots,k_{n-1})$. Let $e$ be the only edge between $T'$ and $S$ in $T$ and let $S'$ be the spider of $T'$ that contains an endvertex of $e$. Furthermore, we may suppose that $S'=S^{k_{n-1}}$. Denote by $v$ and $v'$ the centers of $S$ and $S'$, respectively, by $v_1,\ldots,v_{k_n}$ and $v_1',\ldots,v_{k_{n-1}}'$ the support vertices of $S$ and $S'$, respectively, and by $u_1,\ldots ,u_{k_n}$ and $u_1',\ldots,u_{k_{n-1}}'$ the leaves of $S$ and $S'$, respectively. Note that since $S$ is a leaf spider in $G_T$, deg$_T(x)=\textrm{deg}_S(x)$ for all $x\in V(S)\setminus \{a\}$, where $a$ is the endvertex of $e$ in $S$.

Suppose first that $|F\cap E(S)|<k_n$. Since $F\cap E(T')$ and $F\cap E(S)$ are EOP sets of $T'$ and $S$, respectively, we get $|F|=|F\cap E(T')|+|F\cap E(S)|+1<\eop(T')+1+k_n$. By applying the induction hypothesis to $\eop(T')$, we get $|F|\leq k_1+\ldots+k_n$, as desired. Suppose now that $|F\cap E(S)|=k_n$, which means $F\cap E(S)$ is a $\eop(S)$-set. Since $e\in F$ and $|F\cap E(S)|=\eop(S)$, we derive from Observation \ref{l:structureOfEOP} that $F$ contains an edge of $S$ adjacent to $e$. (Indeed, this is clear when $F\cap E(S)$ is either the set of pendant edges or the set of  internal edges of $S$, while if $k=2$ and $F\cap E(S)$ contains one internal edge and its adjacent pendant edge in $S$, then note that $e$ cannot be incident with a leaf of $S$ due to condition $(C_2)$.) Consequently, no edge of $F\cap E(S')$ is adjacent to $e$. We need to distinguish two cases.

\textit{Case 1.} $e$ is incident with $v$. As $e\in F$ and $|F\cap E(S)|=k_n$, all internal edges of $S$ are in $F$.

{\it Subcase 1.1.} $e=vv'$. Because all extra edges of $T$ are in $F$, we have deg$_T(v')=k_{n-1}+1$ and deg$_T(v_i')=2$ for any $i\in [k_{n-1}]$. Since $T\in \mathcal{F}$, it follows from ($C_{2}$) that there exists $j\in [k_{n-1}]$ such that deg$_T(u_j')=1$. Moreover, since $F$ is an EOP set and $e,vv_1,\ldots,vv_{k_n}\in F$, we deduce that $v'v_i',v_i'u_i'\notin F$ for each $i\in[k_{n-1}].$ Hence, $F'=F\setminus(\{e,vv_1,\ldots,vv_{k_n}\})\cup(\{v_j'u_j',v_1u_1,\ldots,v_{k_n}u_{k_n}\})$ is a $\eop(T)$-set that does not contain all extra edges of $T$, a contradiction.
 
{\it Subcase 1.2.} $e=vv_j'$ for some $j \in [k_{n-1}]$. Because $F$ is an EOP set containing all extra edges of $T$, we have deg$_T(v')=k_{n-1}$, deg$_T(v_j')=3$ and deg$_T(u_j')=1$. Therefore, $F'=F\setminus(\{e,vv_1,\ldots,vv_{k_n}\})\cup(\{v_j'u_j', v_1u_1,\ldots,v_{k_n}u_{k_n}\})$ is a $\eop(T)$-set that does not contain all extra edges of $T$, a contradiction.

{\it Subcase 1.3.} $e=vu_j'$ for some $j\in[k_{n-1}].$ Since $F$ contains all extra edges of $T$, it follows that $\deg_T(v_j')=\deg_T(u_j')=2$. Since $e\in F$ and $v$ is incident with an edge from $F$ different from $e$, it holds that $u_j'v_j',v_j'v'\notin F$. From ($C_{2}$), it follows that $k_{n-1}\geq 3$ and that there exist $s,t\in[k_{n-1}]$ such that deg$_T(u_s')=\textrm{deg}_T(u_t')=1$. Note that $T''=T'-\{u_j',v_j'\}$ is also in family $\mathcal{F}$ \big{(}$T''=T''(k_1,\ldots ,k_{n-2},k_{n-1}-1)$\big{)} as $S'$ still has two leaves $u_s',u_t'$ that are not incident with any extra edge. Since $F\cap E(T')$ is an EOP set of the tree $T''$, it follows from induction hypothesis that $|F \cap E(T'')|\leq k_1+\ldots+k_{n-2}+k_{n-1}-1$. Thus, $|F|=|F\cap E(S)|+1+|F\cap E(T'')|\leq k_1+\ldots+k_n$ as desired. 

\textit{Case 2.} $e$ is not incident with $v$. Then, it follows from ($C_{1}$) that $v'$ is an endvertex of $e$ and hence there exists $i \in [k_n]$ such that $e$ is either $v'v_i$ or $v'u_i$. Since $e\in F$ and $F$ contains an edge in $S$ adjacent to $e$, it follows that $F\cap E(S')=\emptyset$. Because all extra edges of $T$ are in $F$, we deduce that no extra edge is incident with $v_j'$ for any $j\in [k_{n-1}]$. On the other hand, ($C_{2}$) implies that there exists $j\in[k_{n-1}]$ such that deg$_T(u_j')=1$. Hence, $F'=(F\setminus \{e\})\cup \{u_j'v_j'\}$ is a $\eop(T)$-set not having the extra edge $e$, which is impossible.
\end{proof}

Recall that a tree is a \textit{wounded spider} if it can be obtained from a nontrivial star $K_{1,t}$, where $t\geq1$, by subdividing at most $t-1$ of its edges exactly once.

\begin{lemma}\label{l:spider}
If $\nu_I\big{(}S(\ell_1,\ldots,\ell_k)\big{)}=\rho_e^o\big{(}S(\ell_1,\ldots,\ell_k)\big{)}$, then $S(\ell_1,\ldots,\ell_k)\in \{P_{2},P_{5}\}$ or $S(\ell_1,\ldots,\ell_k)$ is a spider with $k\geq3$.
\end{lemma}
\begin{proof}
Let $\ell_1\leq \ldots \leq \ell_k$. We root $T=S(\ell_1,\ldots,\ell_k)$ at the center $x$ and denote the vertices at distance $i$ from $x$ by $v_i^j$ for all $j\in[k]$. If $k\in \{1,2\}$, then $T\cong P_{n}$ for some $n\geq2$. Then, it follows from Propositions \ref{Pro1} and \ref{Pro2} that $n\in \{2,5\}$. So, we may assume that $k\geq3$. Moreover, let $t_{1},t_{2},\ldots,t_{\ell_{k}}$ be the number of the paths on $1,2,\ldots,\ell_{k}$ vertices, respectively. Clearly, $\ell_{k}$ must be at least $2$. We distinguish the following cases depending on $\ell_{k}$.

\textit{Case 1.} $\ell_{k}=2$. Then, $T$ is either a spider or a wounded spider. If $T$ is a spider, we are done. So, let $T$ be a wounded spider. In particular, we have $\ell_1=1$. We then immediately infer that $\nu_{I}(T)\leq k-1<k=\rho_{e}^{o}(T)$, a contradiction.

\textit{Case 2.} $\ell_{k}=3$. If $t_{1}=0$ (or, equivalently $\ell_{1}\geq2$), then the two edges from $P_{\ell_{k}}$ and one edge from any other leg form an EOP set of $T$ of cardinality $k+1$, while any induced matching contains at most one edge from each leg. Therefore, $\rho_{e}^{o}(T)\geq k+1>k=\nu_{I}(T)$, which is impossible. So, it happens that $t_{1}\geq1$. Let $t_{1}\geq2$. Then, any induced matching contains at most one edge from each leg with at least two edges and one edge from the $t_{1}$ pendant edges incident with $x$. Hence, $\nu_{I}(T)=1+k-t_{1}<k\leq \rho_{e}^{o}(T)$, a contradiction. Now let $t_{1}=1$. If $t_{2}\geq1$, then any induced matching of $T$ contains at most one pendant edge from the first two shortest legs and one edge from any other leg. Hence, $\nu_{I}(T)=1+k-2<k\leq \rho_{e}^{o}(T)$. If $t_{2}=0$, then the edges from the paths on three vertices form an EOP set of $T$ of cardinality $2(k-1)$. Hence, $\rho_{e}^{o}(T)\geq2(k-1)>k=\nu_{I}(T)$. Either case leads to a contradiction.  

\textit{Case 3.} $\ell_{k}=4$. Clearly, $\rho_{e}^{o}(T)\geq k+1$. If $\ell_{i}\geq3$ for each $i\in[k]$, then any $\nu_{I}(T)$-set has one edge from each path and one edge from the edges incident with $x$, while $\rho_{e}^{o}(T)=2k$ by a similar argument. So, $\rho_{e}^{o}(T)>k+1=\nu_{I}(T)$. Therefore, $\ell_{i}\leq2$ for some $i\in[k-1]$. Let $Q$ be a $\nu_{I}(T)$-set. If $Q$ contains one edge from the edges incident with $x$, then it has no edge from the $t_{1}+t_{2}$ paths on at most two vertices. Moreover, it has at most one edge from any other path. Thus, we get $\nu_{I}(T)\leq1+k-t_{1}-t_{2}\leq k<\rho_{e}^{o}(T)$. On the other hand, if $Q$ has no edge from the edges incident with $x$, it is readily seen that $\nu_{I}(T)\leq k<\rho_{e}^{o}(T)$. So, we have a contradiction in either case.
    
\textit{Case 4.} $\ell_{k}\geq5$. Let $Q$ be a $\nu_{I}(T)$-set having no edge incident with $x$. Then, $\nu_{I}(T)=\sum_{i=1}^{k}\nu_{I}(P_{\ell_{i}})\leq \sum_{i=1}^{k}\rho_{e}^{o}(P_{\ell_{i}})\leq \rho_{e}^{o}(T)$. Since $\im(T)=\eop(T)$, this implies that $\nu_{I}(P_{\ell_{i}})=\rho_{e}^{o}(P_{\ell_{i}})$ for each $i\in[k]$. Therefore, $\ell_{i}\in \{1,2,5\}$ for each $i\in[k]$. In such a situation, all edges incident with $x$ along with the last two edges from each longest leg form an EOP set of $T$ of cardinality $k+2t_{5}>t_{2}+2t_{5}=\nu_{I}(T)$. Finally, let $Q'$ be a $\nu_{I}(T)$-set having one edge incident with $x$. It is then readily seen that $\nu_{I}(T)\leq1+\sum_{i:\ell_{i}\geq3}\nu_{I}(P_{\ell_{i}-1})$. Let $A$ consist of all edges incident with $x$ along with the edges of a $\rho_{e}^{o}(P_{\ell_{i}}-\{v_{1}^{\ell_{i}},v_{2}^{\ell_{i}}\})$-set for each $i$ with $\ell_{i}\geq4$. Note that $A$ is an EOP set of $T$ of cardinality $k+\sum_{i:\ell_{i}\geq4}\rho_{e}^{o}(P_{\ell_{i}-2})$. On the other hand, $\rho_{e}^{o}(P_{n-1})\geq \nu_{I}(P_{n})$ for all $n\geq3$. With the above observations in mind, we get
\begin{equation}\label{Equ-EOP}
\nu_{I}(T)\leq1+\sum_{i:\ell_{i}\geq3}\nu_{I}(P_{\ell_{i}-1})=1+\sum_{i:\ell_{i}\geq4}\nu_{I}(P_{\ell_{i}-1})+\sum_{i:\ell_{i}=3}\nu_{I}(P_{\ell_{i}-1})\leq \sum_{i:\ell_{i}\geq4}\rho_{e}^{o}(P_{\ell_{i}-2})+1+|\{i\mid \ell_{i}=3\}|.
\end{equation}
Because $\ell_{k}\geq5$, we have $k\geq 1+|\{i\mid \ell_{i}=3\}|$. Furthermore, equality holds if and only if $\ell_{1}=\ldots=\ell_{k-1}=3$. If this is the case, then 
\begin{center}
$\nu_{I}(T)\leq1+|\{i\mid \ell_{i}=3\}|+\nu_{I}(P_{\ell_{k}})<2|\{i\mid \ell_{i}=3\}|+\rho_{e}^{o}(P_{\ell_{k}-1})\leq \rho_{e}^{o}(T)$,
\end{center}
contradicting the assumption. Consequently, $k>1+|\{i\mid \ell_{i}=3\}|$. In view of this and the inequality chain (\ref{Equ-EOP}), we have $\nu_{I}(T)<\rho_{e}^{o}(T)$, a contradiction.

Altogether, we have proved that $\ell_1=\ldots=\ell_k=2$. This completes the proof.
\end{proof}

\begin{theorem}\label{thm:trees2}
If $T$ is a tree with $\eop(T)=\im(T)$, then $T\in {\mathcal{F}}\cup \{P_1,P_2\}$. 
\end{theorem}
\begin{proof}
Let $T\notin \{P_1,P_2\}$ be a tree with $\eop(T)=\im(T)$. Suppose first that $T$ is a path. Then, $T\cong P_5$ by Lemma \ref{l:spider}, and hence $T\in {\mathcal{F}}$. Thus, we may assume that $\Delta(T)\geq3$. We root $T$ at an arbitrary vertex $r\in V(T)$ of degree at least 3, which is the only vertex of $T$ at depth 0. Based on the distance from $r$, all children of $r$ are at depth 1, and denote by $L_i$ all vertices at depth $i$. 
Now let $x\in V(T)$ be a vertex of degree at least 3 at largest possible depth. If $x=r$, then $T$ has exactly one vertex of degree $k\geq3$ and thus $T$ is isomorphic to $S(\ell_1,\ldots,\ell_k)$ for some $\ell_1,\ldots,\ell_k\in \mathbb{N}$. Then, $T$ is a spider by Lemma \ref{l:spider}, and thus it is in $\mathcal{F}$. Hence, let $x\neq r$. Let $y$ be the parent of $x$ and $e=xy$. Let $T_x$ be the component of $T-e$ that contains $x$ and let $T'$ be the other component of $T-e$. Denote deg$_{T_x}(x)=k$. Since deg$_T(x)\geq3$, it follows that $k\geq 2$. From the choice of $x$, it follows that $T_x$ is isomorphic to $S(n_1,\ldots,n_k)$ with center $x$ for some $n_1,\ldots,n_k\in \mathbb{N}$. For any $i\in[k]$, denote by $x_1^i,\ldots,x_{n_i}^i$ the vertices of $S(n_1,\ldots,n_k)$ that induce a path, where $x_1^i$ is a child of $x$. 

Let $F$ be a $\im(T)$-set. Note that $F$ is also a $\eop(T)$-set as $\eop(T)=\im(T)$. 

\begin{claim}\label{claim1}
$e\notin F$.
\end{claim}
\begin{claimproof}
The proof is by contradiction, thus let us assume that $e\in F$. Then, $F\cap E(T_x)$ does not contain any edge incident with a child of $x$. Thus, $|F\cap E(T_x)|=\sum_{i:n_i\geq 3} \im(P_{n_i-1})$. If there exists $i\in[k]$ with $n_i\in \{1,2\}$, we can construct an EOP set of $T$ of cardinality greater than $|F|=\im(T)$ by adding to $F$ the edge incident with $x$ in $P_{n_i}$, a contradiction. Thus, $n_i\geq3$ for all $i\in[k]$. We now set
\begin{center}
$F^*=\big{(}F\cap E(T-T_x)\big{)}\cup \{xx_{1}^{i}\mid 1\leq i\leq k \}\cup W,$
\end{center}
in which $W$ is a $\eop(T_x-\{x_1^1,\ldots,x_1^k,x_2^1,\ldots,x_2^k\})$-set. Note that $F^{*}$ is an EOP set of $T$. Hence,
\begin{equation}\label{EOP1} 
\eop(T)\geq|F^*|=|F\cap E(T')|+1+k+\sum_{i=1}^k \eop(P_{n_i-2})\geq|F\cap E(T')|+k+1+\sum_{i=1}^{k}\left\lceil\frac{n_i-3}{2}\right\rceil.
\end{equation}
On the other hand,
\begin{equation}\label{IND1} 
\im(T)=|F\cap E(T')|+1+|F\cap E(T_x)|=|F\cap E(T')|+1+\sum_{i=1}^k \im(P_{n_i-1})=|F\cap E(T')|+1+\sum_{i=1}^{k}\left\lfloor \frac{n_i}{3}\right\rfloor.
\end{equation} 
Note that $k+\sum_{i=1}^{k}\lceil\frac{n_i-3}{2}\rceil\geq \sum_{i=1}^{k}\lfloor\frac{n_i}{3}\rfloor$, and the equality holds only when $n_i=3$ for all $i\in[k]$. Now, taking into account the equality $\eop(T)=\im(T)$ together with (\ref{EOP1}) and (\ref{IND1}), we infer that $n_i=3$ for all $i\in[k]$. It is then clear that $|F|=\im(T)=|F\cap E(T')|+1+k$. We now observe that
$$\widetilde{F}=\big{(}F\cap E(T')\big{)}\cup \{x_1^ix_2^i,x_2^ix_3^i\mid i\in[k]\}$$ 
is an EOP set of $T$. Hence, $\eop(T)\geq|\widetilde{F}|=|F\cap E(T')|+2k>\im(T)$, a contradiction.   
\end{claimproof}

Next, we focus on the structure of the subtree $T_x$.
\begin{claim}
Either $T_{x}$ is a spider with $x$ as the center, or $T_x$ is isomorphic to $P_5$ with $x$ as a support vertex.
\end{claim}
\begin{claimproof}
First, we distinguish two cases depending on whether there are edges in $F\cap E(T')$ that are adjacent to $e$.

\textit{Case 1.} $F\cap E(T')$ contains no edge adjacent to $e$. In this case, $|F\cap T_x|=\im(T_x)=\eop(T_x)$. Since $x$ has degree at least 2 in $T_x$, it follows from Lemma \ref{l:spider} that $T_x$ is either a $P_5$ or a spider. Note that if $T_x$ is a spider with at least three legs, then $x$ is its center. Moreover, if $T_x$ is isomorphic to $P_5$, then $x$ can be any non-leaf vertex of $P_5$.

\textit{Case 2.} There is $f\in F\cap E(T')$ adjacent to $e$. Since $F$ is an induced matching, the edges incident with $x$ are not contained in $F$. Hence, $|F|=|F\cap E(T')|+\sum_{i=1}^k \im(P_{n_i})=|F\cap E(T')|+\sum_{i=1}^k \lfloor\frac{n_i+1}{3} \rfloor$. If there exists $i\in[k]$ such that $n_i=1$, then 
\begin{center}
$\big{(}F\setminus \big{\{}f,x_1^jx_2^j,x_2^jx_3^j\mid j\in[k]\setminus \{i\}\big{\}}\big{)}\cup \big{\{}e,xx_1^j\mid j\in[k]\big{\}}$
\end{center} 
is an EOP set of cardinality more than $|F|$ (note that $|F\cap \{x_1^jx_2^j,x_2^jx_3^j\mid j\in[k]\setminus \{i\}\}|\leq k-1$ as $F$ is an induced matching), a contradiction. Thus, $n_i\geq2$ for all $i\in[k]$. Now we set 
\begin{center}
$F^{*}=\big{(}(F\cap E(T'))\setminus \{f\}\big{)}\cup \big{\{}e,xx_{1}^{i}\mid i\in[k]\big{\}}\cup W$,
\end{center}
in which $W$ is a $\eop(T_{x}-\{x,x_1^i,x_2^i\mid i\in[k]\})$-set. Note that $F^{*}$ is an EOP set of $T$, and hence \begin{equation}\label{Chain}
\begin{array}{lcl}
|F|=\eop(T)\geq |F^*|&=&(|F\cap E(T')|-1)+1+k+\sum_{i=1}^k\eop(P_{n_i-2})\\
&\geq&|F\cap E(T')|+k+\sum_{i=1}^{k}\lceil\frac{n_i-3}{2}\rceil\\
&=&|F\cap E(T')|+\sum_{i=1}^{k}\lfloor\frac{n_i}{2}\rfloor.
\end{array}
\end{equation}
 
Since $\eop(T)=\im(T)$, we get $\sum_{i=1}^{k}\lfloor\frac{n_i}{2}\rfloor\leq \sum_{i=1}^{k}\lfloor\frac{n_i+1}{3}\rfloor$, implying that $n_i\in \{2,3,5\}$ for each $i\in[k]$. Note first that $n_i=5$ is impossible as in this case $\eop(P_{n_i-2})=\eop(P_3)=2>\lceil\frac{n_i-3}{2}\rceil$, which in turn leads to 
$|F^*|>|F\cap E(T')|+\sum_{i=1}^k \lfloor \frac{n_i}{2} \rfloor \geq|F\cap E(T')|+\sum_{i=1}^k \lfloor \frac{n_i+1}{3}\rfloor=|F|$, a contradiction. Hence, $n_i\in \{2,3\}$ for any $i\in[k]$. It is then clear that $\im(T)=\im(T')+k$. On the other hand, if $n_i=3$ for some $i\in[k]$, then
\begin{center}
$\big{(}(F\cap E(T')\big{)}\cup \{x_{n_{j}}^{j}x_{n_{j}-1}^{j}\mid j\in[k]\}\cup \{x_1^ix_2^i\}$
\end{center}
is an EOP of $T$ of cardinality greater than $\im(T)$, which is impossible. Hence, $n_i=2$ for each $i\in[k]$, which implies that $T_x$ is a spider with center $x$.
\end{claimproof}

Next, we will derive that the subtree $T'$ is also in $\cal F$ as soon as $|V(T')|>2$. 

\begin{claim}
$\im(T')=\eop(T')$.
\end{claim}
\begin{claimproof} 
Suppose that $\eop(T')>\im(T')$. Let $F'$ be a $\eop(T')$-set. First, let $x$ be the center of the spider $T_x$. Because $e\notin F$, it follows from Observation \ref{ob:ImEOPSpider} that $|F\cap E(T_x)|=k=\nu_{I}(T_{x})$. Moreover, $|F\cap E(T')|=\im(T')$. Hence, $|F|=k+\im(T')$. On the other hand, $F'$ together with all pendant edges of $T_x$ form an EOP set of $T$ of cardinality $|F'|+k>\im(T')+k=|F|$, a contradiction. Hence, $\im(T')=\eop(T')$. Assume now that $T_x\cong P_{5}:x_1x_2x_3x_4x_5$, in which $x=x_{2}$. Since $e\notin F$, $|F\cap E(T_x)|\leq2$ and $|F\cap E(T')|\leq \im(T')$, we deduce that $|F|\leq2+\im(T')$. On the other hand, $F'$ together with $\{x_3x_4,x_4x_5\}$ form an EOP set of $T$ of cardinality $|F'|+2>\im(T')+2\geq|F|$, a contradiction. Hence, $\im(T')=\eop(T')$.
\end{claimproof}

We continue by induction on $|V(T)|$ to prove that $T\in \mathcal{F}\cup \{P_{1},P_{2}\}$. Since $\im(T')=\eop(T')$, we have $T'\in \mathcal{F}\cup \{P_{1},P_{2}\}$ by the induction hypothesis. Note that $T'\neq P_{1}$ due to Lemma \ref{l:spider}. Suppose that $T'\cong P_{2}$. If $T_{x}$ is a spider with $x$ as the center, then $T$ is a spider and thus $T\in \cal F$. Otherwise, we have $T\cong S(3,1,2)$, and hence $\im(T)\neq \eop(T)$ by Lemma \ref{l:spider}, a contradiction. So, let $T'\notin\{ P_{1},P_{2}\}$. Thus, $T'=T'(k_1,\ldots,k_n)$ for some integers $k_{1}\ge 2,\ldots,k_{n}\geq2$. Let $v_1,\ldots,v_n$ be the centers of the corresponding spiders, respectively. Without loss of generality, we may assume that $e$ is an edge between $T_x$ and the spider $S^{k_1}$, and that the partition of $T'=T'(k_1,\ldots,k_n)$ into spiders $S^{k_1},\ldots,S^{k_n}$ is chosen in such a way that $k_1$ is as large as possible among all possible partitions of $T'$ into spiders that satisfy $(C_1)$ and $(C_2)$.

Recall by Claim \ref{claim1} that $e\notin F$. Moreover, since $T'\in \mathcal{F}$, it follows from Lemma \ref{l:im} that $\im(T')=k_1+\ldots+k_n$ and that $T'$ has a unique $\im(T')$-set. Furthermore, since $F\cap E(T')$ and $F\cap E(T_x)$ are induced matchings of $T'$ and $T_x$, respectively, we get $|F\cap E(T')|\leq \im(T')$ and $|F\cap E(T_x)|\leq \im(T_x)$. We now distinguish two cases.

\smallskip

\noindent \textit{Case 1.} $T_{x}\cong P_5$, where $x$ is a support vertex. 

\smallskip

Suppose first that $T$ does not satisfy $(C_{1})$, that is, $y\in V(S^{k_1})\setminus \{v_1\}$. Let $y'\in V(S^{k_1})\setminus \{v_1\}$ be a neighbor of $y$. If $yy'\in F$, then $|F\cap(E(T_{x})\cup \{e\})|=1$. Hence, $|F|\leq k_{1}+\ldots+k_{n}+1$. Otherwise, if $yy'\notin F$, we have $|F\cap E(T')|<\im(T')$ by Lemma \ref{l:im}. So, we again have $|F|\leq k_{1}+\ldots+k_{n}+1$ as $|F\cap(E(T_{x})\cup \{e\})|\leq2$. On the other hand, the $\nu_{I}(T')$-set together with $e$ and one pendant edge of $T_x$ (which is not adjacent to $e$) is an EOP set of $T$ of cardinality greater than $|F|$, a contradiction. Thus, $y=v_{1}$ and $(C_1)$ is satisfied in $T$.

Note that $T'$ fulfills $(C_{2})$ as $T'\in \cal F$. Moreover, it is obvious that both leaves of $T_{x}$ have degree $1$ in $T$. Therefore, $T$ fulfills $(C_{2})$ as well.

\smallskip

\noindent \textit{Case 2.} $x$ is the center of $T_x$. 

\smallskip

Since $T'\in \cal F$, every extra edge in $T'$ has an endvertex, which is the center of a spider. In addition, vertex $x$, as the center of $T_{x}$, is an endvertex of $e$. Thus, all extra edges of $T$ satisfy the desired property, and $T$ satisfies $(C_1)$. 

\begin{claim}
$\im(T)=\im(T_x)+\im(T')=k+k_1+\ldots+k_n$ and $T$ has a unique $\im(T)$-set.
\end{claim}
\begin{claimproof}
Since $e\notin F$, $|F|=|F\cap E(T_x)|+|F\cap E(T')|\leq \im(T_x)+\im(T')\leq k+k_1+\ldots+k_n$. Since $(C_1)$ is satisfied in $T$, the set $A$ of all $k$ pendant edges of $T_x$ and all $k_1+\ldots+k_n$ pendant edges of spiders $S^{k_1},\ldots,S^{k_n}$ in $T'$ forms an induced matching in $T$. 
Thus, $\im(T)\geq k+k_1+\ldots+k_n$, and consequently $\im(T)=|F|=k+k_1+\ldots+k_n$. Suppose that there exists a $\im(T)$-set $F^*$ containing an edge which is not in $A$. Then, it follows from Lemma \ref{l:im} that $|F^*\cap E(T_x)|<k$ or $|F^*\cap E(T')|<k_1+\ldots+k_n$. By Claim \ref{claim1}, $e\notin F^*$ and thus $|F^*|=|F^*\cap E(T_x)|+|F^*\cap E(T')|<k+k_1+\ldots+k_n$, which is a contradiction with $F^*$ being a $\im(T)$-set.
\end{claimproof}

To verify that $T$ satisfies $(C_2)$, first note that all leaves of $T_x$ have degree $1$ in $T$ as well. Since $T'\in \cal F$, for every spider $S^{k_i}$ with $i\in[k]$, there are at least two vertices $u,v\in L(S^{k_i})$ such that $\deg_{T'}(u)=\deg_{T'}(v)=1$. In addition, we have $\deg_T(u)=\deg_{T'}
(u)$ for all $u\in V(T')\setminus V(S^{k_1})$. With this in mind, it remains to prove that there exist at least two vertices $u,v\in L(S^{k_1})$ with $\deg_{T}(u)=\deg_{T}(v)=1$. In fact, this is our next claim, by which the proof of the theorem readily follows. 

\begin{claim}
\label{cl:zadnja}
There exist at least two vertices $u,v\in L(S^{k_1})$ with $\deg_{T}(u)=\deg_{T}(v)=1$.
\end{claim}
\begin{claimproof}
First, we remind the reader that the partition of $T'=T'(k_1,\ldots,k_n)$ into spiders $S^{k_1},\ldots,S^{k_n}$ was chosen in such a way that $k_1$ is as large as possible among all feasible partitions of $T'$ into spiders. 
If $y\in N_{S^{k_1}}[v_{1}]$, then at least two leaves of $S^{k_1}$ have degree $1$ in $T$ because $T'$ fulfills $(C_2)$ and we are done. Thus, let $y=w_1^1$, where $L(S^{k_1})=\{w_1^1,\ldots,w_{k_1}^1\}$. Suppose that the claim is false. This implies that there are precisely two vertices in $L(S^{k_1})$, namely $w_{1}^1$ and $w_{2}^1$, of degree $1$ in $T'$. Note that each $w_i^1$, where $i\neq2$, is incident with an extra edge in $T$ (in particular, $w_1^1$ is incident with $e$). Figure~\ref{fig} depicts an illustration of the tree $T$, where each edge label represents the set to which the edge belongs according to the notation as presented in the proof of this claim. 

We introduce some families of spiders from $\{S^{k_1},\ldots,S^{k_n}\}$, which are in a specific way related to the spider $S^{k_1}$. Let ${\cal S}_0=\{S^{k_1}\}$, and for $j\ge 1$, let
\begin{align*}
{\cal S}_j=\bigl\{S^{k_i}\mid & \textrm{ there exist } S^{k_t}\in{\cal S}_{j-1} \textrm{ and } u\in S(S^{k_i}) \textrm{ with } v_tu\in E(T') \textrm{ such that } w\in L(S^{k_i})\cap N(u)\\
 &\textrm{ has degree 1 in } T \bigr\}.
\end{align*}
It is possible that already ${\cal S}_1=\emptyset$ in which case only ${\cal S}_0$ is non-empty. In any case, since $T$ is finite, there exists the largest integer $m$ such that ${\cal S}_m\ne\emptyset$. Set ${\cal S}=\bigcup_{j=0}^m{{\cal S}_j}$. Clearly, $S^{k_1}\in {\cal S}$. For each $S^{k_i}\in {\cal S}$, we let $L(S^{k_i})=\{w_1^i,w_2^i,\ldots,w_{k_i}^i\}$. 

We claim that for every $S^{k_i}\in {\cal S}\setminus {\cal S}_0$, there exist exactly two vertices in $L(S^{k_i})$ with degree $1$ in $T$. As noted earlier, there exist at least two such vertices for each $S^{k_i}\in {\cal S}\setminus {\cal S}_0$ in $T$, as $T'$ enjoys $(C_2)$. Now, suppose that there exists $S^{k_i}\in {\cal S}_\ell$ having more than two leaves with degree $1$ in $T$, and let $\ell$ be as small as possible under this supposition. Clearly, as noted when considering $S^{k_1}$, we have $\ell\ge1$. We may assume, by the definition of family ${\cal S}_\ell$, that the neighbor $u_1^i$ of the leaf $w_1^i$ in $S^{k_i}\in {\cal S}_\ell$ with deg$_T(w_1^i)=1$ has the neighbor $v_t$ in some $S^{v_t}\in {\cal S}_{\ell-1}$. Now, we can reorganize the partition of spiders in $T'$ by moving $u_1^i$ and $w_1^i$ from the spider $S^{k_i}$ (with the center $v_i$) to the spider $S^{k_t}$ (with the center $v_t$). Note that in the resulting partition, the spider $S^{k_i-1}$ with $v_i$ as its center still has at least two leaves with degree $1$ in $T$, while the spider $S^{k_t+1}$ with $v_t$ as its center, which is in ${\cal S}_{\ell-1}$, has three leaves with degree $1$ in $T$. Since the other spiders in the partition of $T'$ into spiders have not changed, we infer that the new partition of $T'$ into spiders satisfies $(C_1)$ and $(C_2)$. Using the same argument, repeated $\ell-1$ times, we derive that the partition of $T'$ into spiders can be reorganized in such a way that for every $j\in[\ell]$, a leaf of $S^{k_i}$ with degree $1$ in $T$ and its neighbor $u$ are moved from a spider $S^{k_i}$ in ${\cal S}_j$ to the spider $S^{k_t}$ in ${\cal S}_{j-1}$, which correspond to each other via the edge $uv_t$ in the definition of family ${\cal S}_j$. Finally, this yields that there exists a partition of $T'$ into spiders such that the spider with $v_1$ as its center has $k_1+1$ leaves, and this is a contradiction with the assumption that $k_1$ is as large as possible among all partitions of $T'$ into spiders that satisfy $(C_1)$ and $(C_2)$.  

Based on the claim from the previous paragraph, we may assume without loss of generality that in each $S^{k_i}\in {\cal S}\setminus {\cal S}_0$, the leaves $w_1^i$ and $w_2^i$ have degree $1$ in $T$, while the remaining leaves $w_3^i,\ldots,w_{k_i}^i$ of $S^{k_i}$ are incident with an extra edge of $T$. In addition, we may assume that the neighbor $u_1^i$ of $w_1^i$ is adjacent to $v_t$ such that $S^{k_t}\in {\cal S}_{j-1}$, where $S^{k_i}\in {\cal S}_j$. For notational purposes, let $u_2^i$ be the neighbor of $w_2^i$. 

Let $F_1$ be the set of edges incident with $x$ and let $F_1'$ contain all edges in $T$ that are adjacent to the edges in $F_1$. (Note that $F_1'$ contains the pendant edges of $T_x$ and also the edge between $y$ ($=w_{1}^{1}$) and its support vertex in $S^{k_1}$.) Clearly, $|F_1|=|F_1'|$. 

As mentioned earlier, $w_{2}^{1}$ is the only leaf of $S^{k_1}$ that has degree $1$ in $T$, $w_{1}^{1}$ is incident with the extra edge $e$, while if $k_1>2$, the leaves $w_{3}^{1},\ldots,w_{k_1}^{1}$ are all incident with at least one extra edge. 
For all $j\in\{3,\ldots,k_1\}$, let $v_{s_j}$ be an arbitrarily chosen neighbor of $w_j^1$. Notice that $w_j^1v_{s_j}$ is an extra edge in $T'$, and that every $v_{s_j}$ is the center of a spider since $T'\in\cal F$. For each $j\in \{3,\ldots,k_1\}$, let $F_{k_1,j}$ be the set of edges incident with $v_{s_j}$ whose other endvertices are not the centers of the spiders, and let $F_{k_1,j}'$ be the set of edges in $F$ that are adjacent to the edges in $F_{k_1,j}$. In particular, $F_{k_1,j}'$ contains the edge between $w_{j}^{1}$ and its support vertex in $S^{k_1}$, and also, all the pendant edges of $S^{k_{s_j}}$ for each $j\in\{3,\ldots,k_1\}$. In addition, it is possible that $F_{k_1,j}$ contains extra edges of $T'$ adjacent to some edges in $F$ from the other spiders. Clearly, $|F_{k_1,j}|=|F_{k_1,j}'|$ for all $j\in \{3,\ldots,k_1\}$. 

In a similar way as in the previous paragraph, we deal with the leaves $w_3^i,\ldots,w_{k_i}^i$ of $S^{k_i}$ for every $S^{k_i}\in {\cal S}\setminus {\cal S}_0$. For all $j\in\{3,\ldots,k_i\}$, let $v_{s_{i,j}}$ be an arbitrarily chosen neighbor of $w_j^i$. Notice that $w_{j}^{i}v_{s_{i,j}}$ is an extra edge in $T'$, and that every $v_{s_{i,j}}$ is the center of a spider since $T'\in\cal F$. For each $j\in \{3,\ldots,k_i\}$, let $F_{k_i,j}$ be the set of edges incident with $v_{s_{i,j}}$ whose other endvertices are not the centers of the spiders, and let $F_{k_i,j}'$ be the set of edges in $F$ that are adjacent to the edges in $F_{k_i,j}$. Clearly, $|F_{k_i,j}|=|F_{k_i,j}'|$ for all $i$, where $S^{k_i}\in {\cal S}\setminus {\cal S}_0$, and for all $j\in \{3,\ldots,k_i\}$. Now, let 
\begin{center}
$F_0=\bigcup_{S^{k_i}\in {\cal S}}\bigcup_{j=3}^{k_i}{F_{k_i,j}}\text{ \   and  \  } F_0'=\bigcup_{S^{k_i}\in {\cal S}}\bigcup_{j=3}^{k_i}{F'_{k_i,j}},$
\end{center}
and note that $|F_0|=|F_0'|$.

\begin{figure}[ht]
\centering
\begin{tikzpicture}[scale=0.5, transform shape]
\node [draw, shape=circle] (v_{t}) at (3,0) {};
\node [draw, shape=circle] (u_{1}^{t}) at (1,-2) {};
\node [draw, shape=circle] (u_{2}^{t}) at (3,-2) {};
\node [draw, shape=circle] (u_{3}^{t}) at (5,-2) {};
\node [draw, shape=circle] (w_{1}^{t}) at (1,-4) {};
\node [draw, shape=circle] (w_{2}^{t}) at (3,-4) {};
\node [draw, shape=circle] (w_{3}^{t}) at (5,-4) {};

\draw (w_{1}^{t})--(u_{1}^{t})--(v_{t})--(u_{2}^{t})--(w_{2}^{t});
\draw (v_{t})--(u_{3}^{t})--(w_{3}^{t});

\node [scale=1.8] at (3.2,0.45) {\large $v_{t}$};
\node [scale=1.8] at (1.6,-2) {\large $u_{1}^{t}$};
\node [scale=1.8] at (1.65,-4) {\large $w_{1}^{t}$};
\node [scale=1.8] at (3.65,-4) {\large $w_{2}^{t}$};
\node [scale=1.8] at (5.65,-4) {\large $w_{3}^{t}$};
\node [scale=1.5] at (5.36,-3) {\large $F_{0}'$};
\node [scale=1.8] at (0.1,0.45) {\large ${\mathcal{S}}_1$};

\node [scale=1.8] at (3.5,0.8) {\large $\ldots\ldots\ldots\ldots\ldots\ldots$};
\node [scale=1.8] at (6.4,0.5) {\large $\vdots$};
\node [scale=1.8] at (6.4,-0.24) {\large $\vdots$};
\node [scale=1.8] at (6.4,-0.98) {\large $\vdots$};
\node [scale=1.8] at (6.4,-1.72) {\large $\vdots$};
\node [scale=1.8] at (6.4,-2.46) {\large $\vdots$};
\node [scale=1.8] at (6.4,-3.2) {\large $\vdots$};
\node [scale=1.8] at (6.4,-3.94) {\large $\vdots$};
\node [scale=1.8] at (3.5,-4.6) {\large $\ldots\ldots\ldots\ldots\ldots\ldots$};
\node [scale=1.8] at (0.6,0.5) {\large $\vdots$};
\node [scale=1.8] at (0.6,-0.24) {\large $\vdots$};
\node [scale=1.8] at (0.6,-0.98) {\large $\vdots$};
\node [scale=1.8] at (0.6,-1.72) {\large $\vdots$};
\node [scale=1.8] at (0.6,-2.46) {\large $\vdots$};
\node [scale=1.8] at (0.6,-3.2) {\large $\vdots$};
\node [scale=1.8] at (0.6,-3.94) {\large $\vdots$};


\node [draw, shape=circle] (a0) at (6,-7) {};
\node [draw, shape=circle] (a1) at (4,-9) {};
\node [draw, shape=circle] (a2) at (6,-9) {};
\node [draw, shape=circle] (a3) at (8,-9) {};
\node [draw, shape=circle] (a11) at (4,-11) {};
\node [draw, shape=circle] (a22) at (6,-11) {};
\node [draw, shape=circle] (a33) at (8,-11) {};

\draw (a11)--(a1)--(a0)--(a2)--(a22);
\draw (a0)--(a3)--(a33);
\draw (a0)--(w_{3}^{t});

\node [scale=1.5] at (4.3,-8.1) {\large $F_{0}$};
\node [scale=1.5] at (6.35,-8.1) {\large $F_{0}$};
\node [scale=1.5] at (7.65,-8.1) {\large $F_{0}$};
\node [scale=1.5] at (6,-5.8) {\large $F_{0}$};

\node [scale=1.5] at (4.35,-10) {\large $F_{0}'$};
\node [scale=1.5] at (6.35,-10) {\large $F_{0}'$};
\node [scale=1.5] at (8.35,-10) {\large $F_{0}'$};


\node [draw, shape=circle] (v_{i}) at (14,-3) {};
\node [draw, shape=circle] (u_{1}^{i}) at (11,-5) {};
\node [draw, shape=circle] (u_{2}^{i}) at (13,-5) {};
\node [draw, shape=circle] (u_{3}^{i}) at (15,-5) {};
\node [draw, shape=circle] (u_{4}^{i}) at (17,-5) {};
\node [draw, shape=circle] (w_{1}^{i}) at (11,-7) {};
\node [draw, shape=circle] (w_{2}^{i}) at (13,-7) {};
\node [draw, shape=circle] (w_{3}^{i}) at (15,-7) {};
\node [draw, shape=circle] (w_{4}^{i}) at (17,-7) {};

\draw (w_{1}^{i})--(u_{1}^{i})--(v_{i})--(u_{2}^{i})--(w_{2}^{i});
\draw (w_{3}^{i})--(u_{3}^{i})--(v_{i})--(u_{4}^{i})--(w_{4}^{i});
\draw (v_{t})--(u_{1}^{i});

\node [scale=1.8] at (14,-2.55) {\large $v_{i}$};
\node [scale=1.8] at (11.65,-5) {\large $u_{1}^{i}$};
\node [scale=1.5] at (15.34,-6) {\large $F_{0}'$};
\node [scale=1.5] at (17.34,-6) {\large $F_{0}'$};
\node [scale=1.8] at (11.65,-7) {\large $w_{1}^{i}$};
\node [scale=1.8] at (13.65,-7) {\large $w_{2}^{i}$};
\node [scale=1.8] at (15.65,-7) {\large $w_{3}^{i}$};
\node [scale=1.8] at (17.65,-7) {\large $w_{4}^{i}$};
\node [scale=1.8] at (10,-2.5) {\large ${\mathcal{S}}_2$};

\node [scale=1.8] at (14.45,-2.1) {\large $\ldots\ldots\ldots\ldots\ldots\ldots\ldots\ldots$};
\node [scale=1.8] at (18.35,-2.4) {\large $\vdots$};
\node [scale=1.8] at (18.35,-3.14) {\large $\vdots$};
\node [scale=1.8] at (18.35,-3.88) {\large $\vdots$};
\node [scale=1.8] at (18.35,-4.62) {\large $\vdots$};
\node [scale=1.8] at (18.35,-5.36) {\large $\vdots$};
\node [scale=1.8] at (18.35,-6.1) {\large $\vdots$};
\node [scale=1.8] at (18.35,-6.84) {\large $\vdots$};
\node [scale=1.8] at (18.35,-7.49) {\large $.$};
\node [scale=1.8] at (14.45,-7.73) {\large $\ldots\ldots\ldots\ldots\ldots\ldots\ldots\ldots$};
\node [scale=1.8] at (10.6,-2.4) {\large $\vdots$};
\node [scale=1.8] at (10.6,-3.14) {\large $\vdots$};
\node [scale=1.8] at (10.6,-3.88) {\large $\vdots$};
\node [scale=1.8] at (10.6,-4.62) {\large $\vdots$};
\node [scale=1.8] at (10.6,-5.36) {\large $\vdots$};
\node [scale=1.8] at (10.6,-6.1) {\large $\vdots$};
\node [scale=1.8] at (10.6,-6.84) {\large $\vdots$};
\node [scale=1.8] at (10.6,-7.49) {\large $.$};


\node [draw, shape=circle] (b0) at (13,-11) {};
\node [draw, shape=circle] (b1) at (12,-13) {};
\node [draw, shape=circle] (b2) at (14,-13) {};
\node [draw, shape=circle] (b11) at (12,-15) {};
\node [draw, shape=circle] (b22) at (14,-15) {};

\draw (b11)--(b1)--(b0)--(b2)--(b22);
\draw (w_{3}^{i})--(b0);

\node [scale=1.5] at (13.45,-9.1) {\large $F_{0}$};
\node [scale=1.5] at (12.03,-12) {\large $F_{0}$};
\node [scale=1.5] at (13.95,-12) {\large $F_{0}$};
\node [scale=1.5] at (12.34,-14) {\large $F_{0}'$};
\node [scale=1.5] at (14.34,-14) {\large $F_{0}'$};


\node [draw, shape=circle] (c0) at (19,-11) {};
\node [draw, shape=circle] (c1) at (17,-13) {};
\node [draw, shape=circle] (c2) at (19,-13) {};
\node [draw, shape=circle] (c3) at (21,-13) {};
\node [draw, shape=circle] (c11) at (17,-15) {};
\node [draw, shape=circle] (c22) at (19,-15) {};
\node [draw, shape=circle] (c33) at (21,-15) {};

\draw (c11)--(c1)--(c0)--(c2)--(c22);
\draw (c0)--(c3)--(c33);
\draw (c0)--(w_{4}^{i});

\node [scale=1.5] at (18.5,-9.13) {\large $F_{0}$};
\node [scale=1.5] at (17.38,-12) {\large $F_{0}$};
\node [scale=1.5] at (19.35,-12) {\large $F_{0}$};
\node [scale=1.5] at (20.6,-12) {\large $F_{0}$};
\node [scale=1.5] at (17.34,-14) {\large $F_{0}'$};
\node [scale=1.5] at (19.34,-14) {\large $F_{0}'$};
\node [scale=1.5] at (21.34,-14) {\large $F_{0}'$};


\node [draw, shape=circle] (d0) at (10,4) {};
\node [draw, shape=circle] (d1) at (8,2) {};
\node [draw, shape=circle] (d2) at (10,2) {};
\node [draw, shape=circle] (d3) at (12,2) {};
\node [draw, shape=circle] (d11) at (8,0) {};
\node [draw, shape=circle] (d22) at (10,0) {};
\node [draw, shape=circle] (d33) at (12,0) {};

\draw (d11)--(d1)--(d0)--(d2)--(d22);
\draw (d0)--(d3)--(d33);
\draw (v_{t})--(d11);

\node [scale=1.5] at (8.5,3) {\large $L$};
\node [scale=1.5] at (10.27,3) {\large $L$};
\node [scale=1.5] at (11.5,3) {\large $L$};
\node [scale=1.5] at (8.33,1) {\large $L'$};
\node [scale=1.5] at (10.33,1) {\large $L'$};
\node [scale=1.5] at (12.33,1) {\large $L'$};


\node [draw, shape=circle] (e0) at (18,8) {};
\node [draw, shape=circle] (e1) at (16,6) {};
\node [draw, shape=circle] (e2) at (18,6) {};
\node [draw, shape=circle] (e3) at (20,6) {};
\node [draw, shape=circle] (e11) at (16,4) {};
\node [draw, shape=circle] (e22) at (18,4) {};
\node [draw, shape=circle] (e33) at (20,4) {};

\draw (e11)--(e1)--(e0)--(e2)--(e22);
\draw (e0)--(e3)--(e33);
\draw (d0)--(e1);

\node [scale=1.5] at (16.34,5) {\large $L'$};
\node [scale=1.5] at (12,5.1) {\large $L$};


\node [draw, shape=circle] (v_{1}) at (-4,-5) {};
\node [draw, shape=circle] (u_{1}^{1}) at (-7,-7) {};
\node [draw, shape=circle] (u_{2}^{1}) at (-5,-7) {};
\node [draw, shape=circle] (u_{3}^{1}) at (-3,-7) {};
\node [draw, shape=circle] (u_{4}^{1}) at (-1,-7) {};
\node [draw, shape=circle] (w_{1}^{1}) at (-7,-9) {};
\node [draw, shape=circle] (w_{2}^{1}) at (-5,-9) {};
\node [draw, shape=circle] (w_{3}^{1}) at (-3,-9) {};
\node [draw, shape=circle] (w_{4}^{1}) at (-1,-9) {};

\draw (w_{1}^{1})--(u_{1}^{1})--(v_{1})--(u_{2}^{1})--(w_{2}^{1});
\draw (w_{3}^{1})--(u_{3}^{1})--(v_{1})--(u_{4}^{1})--(w_{4}^{1});
\draw (v_{1})--(u_{1}^{t});

\node [scale=1.8] at (-4.35,-4.55) {\large $v_{1}$};
\node [scale=1.8] at (-6.35,-7.1) {\large $u_{1}^{1}$};
\node [scale=1.8] at (-6.35,-9) {\large $w_{1}^{1}$};
\node [scale=1.8] at (-4.35,-9) {\large $w_{2}^{1}$};
\node [scale=1.8] at (-2.35,-9) {\large $w_{3}^{1}$};
\node [scale=1.8] at (-0.35,-9) {\large $w_{4}^{1}$};
\node [scale=1.5] at (-7.38,-8) {\large $F_{1}'$};
\node [scale=1.8] at (-8.3,-4.5) {\large ${\mathcal{S}}_0$};
\node [scale=1.5] at (-2.65,-8) {\large $F_{0}'$};
\node [scale=1.5] at (-0.65,-8) {\large $F_{0}'$};

\node [scale=1.8] at (-3.75,-4.1) {\large $\ldots\ldots\ldots\ldots\ldots\ldots\ldots\ldots$};
\node [scale=1.8] at (0.15,-4.4) {\large $\vdots$};
\node [scale=1.8] at (0.15,-5.14) {\large $\vdots$};
\node [scale=1.8] at (0.15,-5.88) {\large $\vdots$};
\node [scale=1.8] at (0.15,-6.62) {\large $\vdots$};
\node [scale=1.8] at (0.15,-7.36) {\large $\vdots$};
\node [scale=1.8] at (0.15,-8.1) {\large $\vdots$};
\node [scale=1.8] at (0.15,-8.84) {\large $\vdots$};
\node [scale=1.8] at (0.15,-9.49) {\large $.$};
\node [scale=1.8] at (-3.75,-9.73) {\large $\ldots\ldots\ldots\ldots\ldots\ldots\ldots\ldots$};
\node [scale=1.8] at (-7.74,-4.4) {\large $\vdots$};
\node [scale=1.8] at (-7.74,-5.14) {\large $\vdots$};
\node [scale=1.8] at (-7.74,-5.88) {\large $\vdots$};
\node [scale=1.8] at (-7.74,-6.62) {\large $\vdots$};
\node [scale=1.8] at (-7.74,-7.36) {\large $\vdots$};
\node [scale=1.8] at (-7.74,-8.1) {\large $\vdots$};
\node [scale=1.8] at (-7.74,-8.84) {\large $\vdots$};
\node [scale=1.8] at (-7.74,-9.49) {\large $.$};


\node [draw, shape=circle] (x) at (-9,-11) {};
\node [draw, shape=circle] (f1) at (-10,-13) {};
\node [draw, shape=circle] (f2) at (-8,-13) {};
\node [draw, shape=circle] (f11) at (-10,-15) {};
\node [draw, shape=circle] (f22) at (-8,-15) {};

\draw (f11)--(f1)--(x)--(f2)--(f22);
\draw (x)--(w_{1}^{1});

\node [scale=1.8] at (-9.4,-10.7) {\large $x$};
\node [scale=1.5] at (-8.55,-9.9) {\large $F_{1}$};
\node [scale=1.5] at (-9.97,-12) {\large $F_{1}$};
\node [scale=1.5] at (-8.05,-12) {\large $F_{1}$};
\node [scale=1.5] at (-9.63,-14) {\large $F_{1}'$};
\node [scale=1.5] at (-7.63,-14) {\large $F_{1}'$};


\node [draw, shape=circle] (v_{s_3}) at (-4,-11) {};
\node [draw, shape=circle] (g1) at (-6,-13) {};
\node [draw, shape=circle] (g2) at (-4,-13) {};
\node [draw, shape=circle] (g3) at (-2,-13) {};
\node [draw, shape=circle] (g11) at (-6,-15) {};
\node [draw, shape=circle] (g22) at (-4,-15) {};
\node [draw, shape=circle] (g33) at (-2,-15) {};

\draw (g11)--(g1)--(v_{s_3})--(g2)--(g22);
\draw (v_{s_3})--(g3)--(g33);
\draw (v_{s_3})--(w_{3}^{1});

\node [scale=1.5] at (-3.3,-10.45) {\large $F_{0}$};
\node [scale=1.8] at (-4.65,-10.9) {\large $v_{s_3}$};
\node [scale=1.5] at (-5.65,-12) {\large $F_{0}$};
\node [scale=1.5] at (-3.63,-12) {\large $F_{0}$};
\node [scale=1.5] at (-2.48,-12) {\large $F_{0}$};
\node [scale=1.5] at (-5.66,-14) {\large $F_{0}'$};
\node [scale=1.5] at (-3.66,-14) {\large $F_{0}'$};
\node [scale=1.5] at (-1.66,-14) {\large $F_{0}'$};


\node [draw, shape=circle] (v_{s_4}) at (2,-14) {};
\node [draw, shape=circle] (h1) at (0,-16) {};
\node [draw, shape=circle] (h2) at (2,-16) {};
\node [draw, shape=circle] (h3) at (4,-16) {};
\node [draw, shape=circle] (h11) at (0,-18) {};
\node [draw, shape=circle] (h22) at (2,-18) {};
\node [draw, shape=circle] (h33) at (4,-18) {};

\draw (w_{4}^{1})--(v_{s_4});
\draw (h11)--(h1)--(v_{s_4})--(h2)--(h22);
\draw (v_{s_4})--(h3)--(h33);

\node [scale=1.8] at (1.3,-13.7) {\large $v_{s_4}$};
\node [scale=1.5] at (0.4,-15) {\large $F_{0}$};
\node [scale=1.5] at (2.34,-15) {\large $F_{0}$};
\node [scale=1.5] at (3.53,-15) {\large $F_{0}$};
\node [scale=1.5] at (1.25,-12) {\large $F_{0}$};
\node [scale=1.5] at (0.33,-17) {\large $F_{0}'$};
\node [scale=1.5] at (2.33,-17) {\large $F_{0}'$};
\node [scale=1.5] at (4.33,-17) {\large $F_{0}'$};


\node [draw, shape=circle] (i0) at (8,-13) {};
\node [draw, shape=circle] (i1) at (6,-15) {};
\node [draw, shape=circle] (i2) at (8,-15) {};
\node [draw, shape=circle] (i3) at (10,-15) {};
\node [draw, shape=circle] (i11) at (6,-17) {};
\node [draw, shape=circle] (i22) at (8,-17) {};
\node [draw, shape=circle] (i33) at (10,-17) {};

\draw (i11)--(i1)--(i0)--(i2)--(i22);
\draw (i0)--(i3)--(i33);
\draw (v_{s_4})--(i1);

\node [scale=1.5] at (6.33,-16) {\large $F_{0}'$};
\node [scale=1.5] at (4.25,-14.23) {\large $F_{0}$};


\node [draw, shape=circle] (v_{j}) at (-5,8) {};
\node [draw, shape=circle] (u_{1}^{j}) at (-7,6) {};
\node [draw, shape=circle] (u_{2}^{j}) at (-5,6) {};
\node [draw, shape=circle] (u_{3}^{j}) at (-3,6) {};
\node [draw, shape=circle] (w_{1}^{j}) at (-7,4) {};
\node [draw, shape=circle] (w_{2}^{j}) at (-5,4) {};
\node [draw, shape=circle] (w_{3}^{j}) at (-3,4) {};

\draw (w_{1}^{j})--(u_{1}^{j})--(v_{j})--(u_{2}^{j})--(w_{2}^{j});
\draw (v_{j})--(u_{3}^{j})--(w_{3}^{j});
\draw (v_{t})--(u_{3}^{j});

\node [scale=1.8] at (-4.5,8.2) {\large $v_{j}$};
\node [scale=1.5] at (-3.44,5) {\large $M'$};
\node [scale=1.8] at (-6.35,4) {\large $w_{1}^{j}$};
\node [scale=1.8] at (-4.35,4) {\large $w_{2}^{j}$};
\node [scale=1.8] at (-2.35,4) {\large $w_{3}^{j}$};


\node [draw, shape=circle] (j0) at (-8,1) {};
\node [draw, shape=circle] (j1) at (-10,-1) {};
\node [draw, shape=circle] (j2) at (-8,-1) {};
\node [draw, shape=circle] (j3) at (-6,-1) {};
\node [draw, shape=circle] (j11) at (-10,-3) {};
\node [draw, shape=circle] (j22) at (-8,-3) {};
\node [draw, shape=circle] (j33) at (-6,-3) {};

\draw (j11)--(j1)--(j0)--(j2)--(j22);
\draw (j0)--(j3)--(j33);
\draw (j0)--(w_{3}^{j});

\node [scale=1.5] at (-6.6,2.3) {\large $M$};
\node [scale=1.5] at (-9.75,-0.2) {\large $M$};
\node [scale=1.5] at (-7.64,-0.2) {\large $M$};
\node [scale=1.5] at (-6.22,-0.2) {\large $M$};
\node [scale=1.5] at (-9.57,-2) {\large $M'$};
\node [scale=1.5] at (-7.57,-2) {\large $M'$};
\node [scale=1.5] at (-5.57,-2) {\large $M'$};

\end{tikzpicture}
\caption{{\small An illustration of the situation discussed in Claim \ref{cl:zadnja} of the proof of Theorem \ref{thm:trees2}.}}\label{fig}
\end{figure}
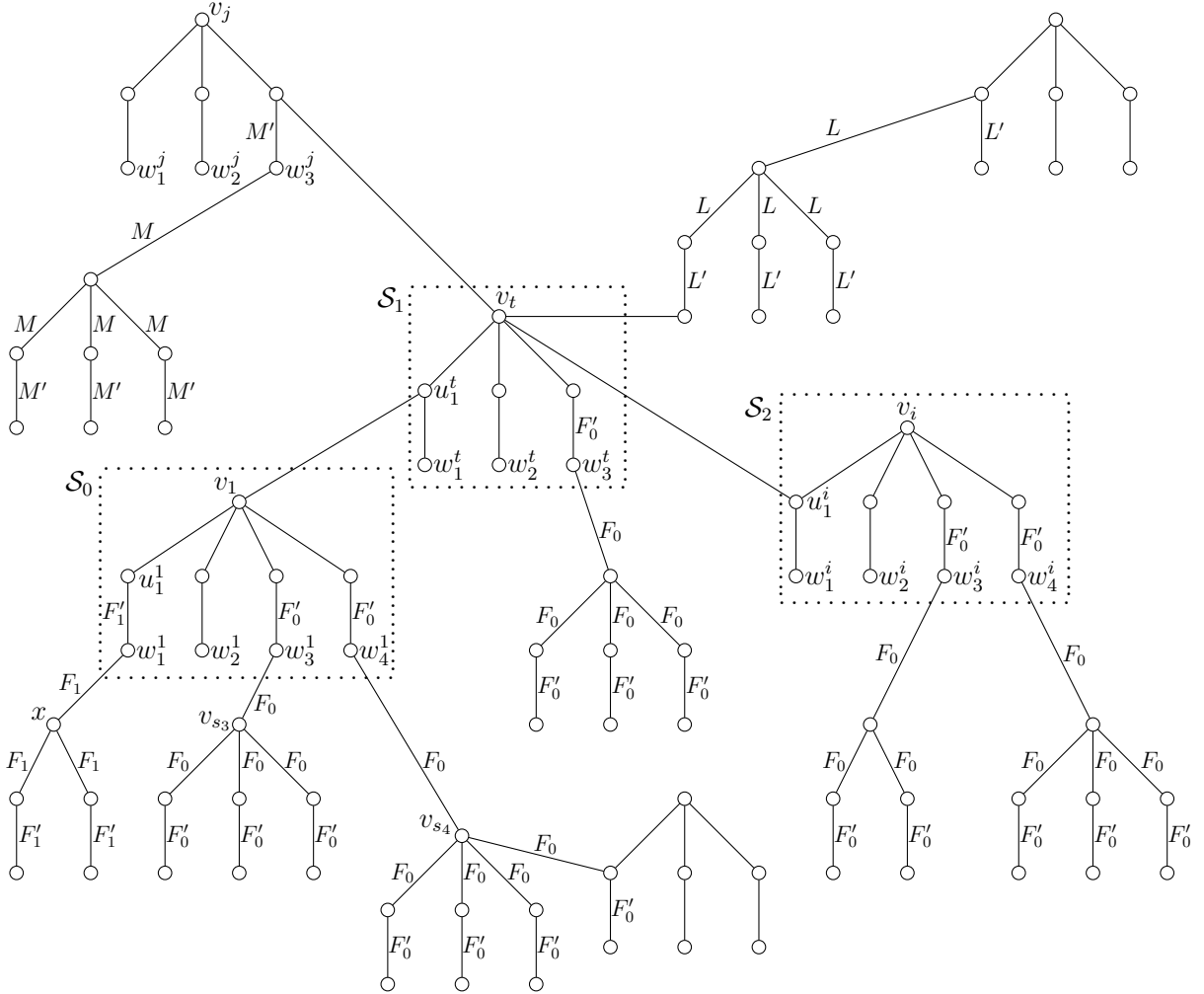

Next, for every $S^{k_i}\in {\cal S}$ (including $S^{k_1}$), let $Z_i$ be the set of neighbors $z$ of $v_i$, which are leaves in the corresponding spider $S^{k_{r_z}}$ to which they belong with respect to the partition of $T'$ into spiders. (Note that every such edge $v_{i}z$ is an extra edge of $T$.) Now, let 
\begin{center}
$L_i=\bigl\{v_{r_z}u\mid z\in Z_i\ \mbox{and}\ u\notin\{v_1,\ldots,v_n\}\bigr\}$
\end{center}
be the set of edges in $T$ incident with $v_{r_z}$, whose other endvertices are not the centers of the spiders in $T'$, and let $L_i'$ be the set of edges in $F$ that are adjacent to an edge in $L_{i}$.
(Note that $L_i$ contains all internal edges of $S^{k_{r_z}}$, but it may also contain extra edges of $T$ if $v_{r_z}$ is adjacent to non-centers of other spiders.) Clearly, $|L_i|=|L_i'|$ for all $i$. In addition, let $L$ be the union of all $L_i$, where $S^{k_i}\in {\cal S}$. Similarly, let $L'$ be the union of all such $L_i'$.

For every $S^{k_i}\in {\cal S}$, let $Q_i$ be the set of vertices $q$ in $L(S^{k_j})$, where $j\ne i$, such that the neighbor $u\in V(S^{k_j})$ of $q$ is adjacent to $v_i$, and $\textrm{deg}_T(q)>1$. That is, every $q\in Q_i$ is incident with an extra edge. Clearly, $S^{k_j}\notin \cal S$ as $\textrm{deg}_T(q)>1$. Now, let $v_{s_q}$ be an arbitrarily chosen neighbor of $q$, different from $u$. Note that $qv_{s_q}$ is an extra edge in $T$, and that every $v_{s_q}$ is the center of a spider since $T'\in\cal F$. Let $M_i$ be the set of edges incident with $v_{s_q}$, over all $q\in Q_i$, such that the other endvertices of these edges are not the centers of the spiders. Let $M_i'$ be the set of edges in $F$ that are adjacent to the edges in $M_i$. Clearly, $|M_i|=|M_i'|$ for all $i$, where $S^{k_i}\in {\cal S}$. Let $M$ be the union of all these sets $M_i$, and similarly $M'$, is defined as the union of all $M_i'$. We now set
$$\widetilde{F}=\Big(F\setminus \big(\hskip -12pt\bigcup_{S^{k_i}\in {\cal S}\setminus\{S^{k_1}\}}\hskip-18pt\{u_1^iw_1^i\}\cup F_1'\cup F_0' \cup L'\cup M'\big)\Big) \bigcup \big(\hskip -12pt\bigcup_{S^{k_i}\in {\cal S}\setminus\{S^{k_1}\}}\hskip-18pt\{v_iu_2^i\}\cup F_1 \cup F_0\cup L\cup M\cup \{v_1u_2^1\} \big{)}.$$
One can readily check that $\widetilde{F}$ is an EOP set in $T$ with $|\widetilde{F}|=|F|+1$, which is impossible. 
\end{claimproof}

As observed earlier, Claim \ref{cl:zadnja} implies that $T$ satisfies $(C_2)$ also in Case 2. We derive that $T\in \cal F$. This completes the proof. 
\end{proof}

Combining Theorems~\ref{thm:trees1} and~\ref{thm:trees2} we derive the desired characterization, which is the main result of this section.

\begin{corollary}
If $T$ is a tree, then $\eop(T)=\im(T)$ if and only if $T\in {\mathcal{F}}\cup \{P_1,P_2\}$. 
\end{corollary}

Based on the nontrivial proof of the above result, a complete characterization of the graphs $G$ for which $\eop(G)=\im(G)$ holds should be difficult. Possibly, obtaining such a characterization within a special class of graphs could already be an interesting problem. 


\section{Lexicographic product}\label{sec-lex}

The lexicographic product provides an environment for induced matching and edge open packing in which their behavior can be described more precisely. In particular, it enables us to give the exact value of induced matching number in this case. 

We start the section by introducing some notations. The edge projection of $G \circ H$ to $G$ is the  mapping ${\rm pr}_G: E(G \circ H) \to E(G) \cup V(G)$ with  ${\rm pr}_{G}((g,h)(g',h'))=gg'$ if $g \neq g'$ and ${\rm pr}_{G}((g,h)(g,h'))=g$.

\begin{theorem}\label{lex}
For any graphs $G$ and $H$, $\nu_{I}(G\circ H)=\alpha(G)\nu_{I}(H)$.
\end{theorem}
\begin{proof}
Let $Q$ be a $\nu_{I}(G\circ H)$-set. We set $Q=Q_{1}\cup Q_{2}$, in which $Q_{1}=\{(g,h)(g',h')\in Q\mid g\neq g'\}$ and $Q_{2}=Q\setminus Q_{1}$. Notice that ${\rm pr}_{G}(Q_{1})=\bigcup_{e\in Q_{1}}{\rm pr}_{G}(e)\subseteq E(G)$ and ${\rm pr}_{G}(Q_{2})=\bigcup_{e\in Q_{2}}{\rm pr}_{G}(e)\subseteq V(G)$. Let $e=(g,h)(g',h')\in Q_{1}$. 

Since $Q$ is an induced matching, it follows from the definition of adjacency in lexicographic product that every edge in ${\rm pr}_G(Q_1)$ is the projection of precisely one edge in $Q_1$ (in particular, $|{\rm pr}_G(Q_1)|=|Q_1|$) and that ${\rm pr}_{G}(Q_{2})\cap V({\rm pr}_{G}\big{(}Q_{1})\big{)}=\emptyset$. Moreover, two edges in ${\rm pr}_G(Q_1)$ cannot have a common endvertex nor a common edge. Thus, ${\rm pr}_G(Q_1)$ is an induced matching in $G$, and so $|{\rm pr}_{G}(Q_{1})|\leq \nu_{I}(G)$. Now let $g$ and $g'$ be arbitrary vertices of ${\rm pr}_G(Q_2)$ and $g_1g_2$ an arbitrary edge in ${\rm pr}_G(Q_1)$. Since $Q$ is an induced matching of $G \circ H$, it follows that $gg' \notin E(G)$ and thus ${\rm pr}_G(Q_2)$ is an independent set of $G$. Moreover, $g$ is not adjacent to any of the vertices $g_1$ and $g_2$.

All in all, we have proved that a subset $J$ containing the vertices in ${\rm pr}_{G}(Q_{2})$ along with one endvertex from each edge of ${\rm pr}_{G}(Q_{1})$ is an independent set in $G$, and in particular, $|{\rm pr}_{G}(Q_{1})|+|{\rm pr}_{G}(Q_{2})|\leq \alpha(G)$. Furthermore, each $H$-fiber with $g\in {\rm pr}_{G}(Q_{2})$, being isomorphic to $H$, has at most $\nu_{I}(H)$ edges from $Q_{2}$. Thus, we get
\begin{equation}\label{IM1}
\begin{array}{lcl}
\nu_{I}(G\circ H)&=&|Q|=|{\rm pr}_{G}(Q_{1})|+|Q_{2}|\\
&\leq&|{\rm pr}_{G}(Q_{1})|+|{\rm pr}_{G}(Q_{2})|\nu_{I}(H)\leq(|{\rm pr}_{G}(Q_{1})|+|{\rm pr}_{G}(Q_{2})|)\nu_{I}(H)\leq \alpha(G)\nu_{I}(H).
\end{array}
\end{equation}

Conversely, let $I$ and $A$ be an $\alpha(G)$-set and a $\nu_{I}(H)$-set, respectively. By the structure of $G\circ H$, $\{(g,h)(g,h')\in E(G\circ H)\mid g\in I\ \mbox{and}\ hh'\in A\}$ is an induced matching in $G\circ H$ of cardinality $\alpha(G)\nu_{I}(H)$. So, $\nu_{I}(G\circ H)\geq \alpha(G)\nu_{I}(H)$. This completes the proof in view of (\ref{IM1}).
\end{proof}

By using similar techniques as in the proof of Theorem \ref{lex}, we bound the EOP number of lexicographic product graphs from above and below. However, there are many different situations that must be taken into account. The bounds are widely sharp. In particular, they are the same for an infinite family of graphs $G$ \big{(}those with $\alpha(G)=\rho_{e}^{o}(G)$\big{)} regardless of our choice of $H$, leading to the exact value of the parameter in this case.

\begin{theorem}\label{thm:lex}
For any graphs $G$ and $H$,
$$\rho_{e}^{o}(G)\alpha(H)\leq \rho_{e}^{o}(G\circ H)\leq \rho_{e}^{o}(G)\alpha(H)+\rho_{e}^{o}(H)\big{(}\alpha(G)-\rho_{e}^{o}(G)\big{)}.$$
These bounds are sharp.
\end{theorem}
\begin{proof}
Let $B$ be a $\rho_{e}^{o}(G\circ H)$-set. Similarly to the proof of Theorem \ref{lex}, we set $B=B_{1}\cup B_{2}$, in which $B_{1}=\{(g,h)(g',h')\in B\mid g\neq g'\}$ and $B_{2}=B\setminus B_{1}$. Recall that every edge in $B_{2}$ is of the form $(g,h)(g,h')$ for some $g\in V(G)$ and distinct vertices $h,h'\in V(H)$.

Suppose that ${\rm pr}_{G}\big{(}(g_{1},h_{1})(g_{1}',h_{1}')\big{)}$ and ${\rm pr}_{G}\big{(}(g_{2},h_{2})(g_{2}',h_{2}')\big{)}$ have a common edge, say $g_{1}g_{2}'$, in $G$ for some $(g_{1},h_{1})(g_{1}',h_{1}'),(g_{2},h_{2})(g_{2}',h_{2}')\in B_{1}$. This implies that $(g_{1},h_{1})(g_{1}',h_{1}')$ and $(g_{2},h_{2})(g_{2}',h_{2}')$ have the common edge $(g_{1},h_{1})(g_{2}',h_{2}')$ in $G\circ H$, contradicting the fact that $B$ is an EOP set in $G\circ H$. The above discussion guarantees that ${\rm pr}_{G}(B_{1})$ forms an EOP set in $G$. With this in mind, the subgraph of $G$ induced by ${\rm pr}_{G}(B_{1})$ is a disjoint union of stars $S_{1},\ldots,S_{k}$ with centers $s_{1},\ldots,s_{k}$, respectively (if $S_{i}\cong K_{2}$ for some $1\leq i\leq k$, then we consider any one of its endvertices as the center). Note that no two vertices from different stars are adjacent in $G$. Analogously to the proof of Theorem \ref{lex}, ${\rm pr}_{G}(B_{2})$ is independent in $G$, ${\rm pr}_{G}(B_{2})\cap\big{(}\bigcup_{i=1}^{k}V(S_{i})\big{)}=\emptyset$ and there is no edge with one endvertex in ${\rm pr}_{G}(B_{2})$ and the other in $\bigcup_{i=1}^{k}V(S_{i})$ in $G$. In fact, all non-central vertices of the stars $S_{1},\ldots,S_{k}$ along with all vertices in ${\rm pr}_{G}(B_{2})$ form an independent set in $G$, In particular, we infer that
\begin{equation}\label{alpha}
|{\rm pr}_{G}(B_{1})|+|{\rm pr}_{G}(B_{2})|\leq \alpha(G).
\end{equation}

We now consider an edge $(g,h)(g',h')\in B_{1}$. Let ${\rm pr}_{G}\big{(}(g,h)(g',h')\big{)}=gg'\in E(S_{i})$. Combining the adjacency rule of $G\circ H$ and the fact that $B$ is an EOP set in $G\circ H$ we infer that all edges $e\in B_{1}$ with ${\rm pr}_{G}(e)=gg'$ have a common endvertex in $\{(g,h),(g',h')\}$. Recalling that $s_{i}$ is the center of the star $S_{i}$, we may assume without loss of generality, renaming the vertices if necessary, that $g=s_{i}$. 
Set
\begin{center}
$A=\big{\{}(s_i,h)(g',h'')\in B_1\mid {\rm pr}_{G}\big{(}(g,h)(g',h'')\big{)}=s_{i}g'\big{\}}$. \end{center}
Since the endvertices of edges of $A$ induce a star in $G\circ H$,  we get $|\{h''\mid (s_i,h)(g',h'')\in A\}|\leq \alpha(H)$. With the above discussion in mind and assuming that $I$ is an $\alpha(H)$-set, it is readily observed that
\begin{center}
$B'=\bigcup_{i=1}^{k}\bigcup_{s_{i}g^*\in E(S_{i})}\bigcup_{h'\in I}\big{\{}(s_{i},h)(g^*,h')\big{\}}$
\end{center}
is an EOP set in $G\circ H$ with $|B'|\geq|B_{1}|$, for which ${\rm pr}_{G}(B')\subseteq {\rm pr}_{G}(B_{1})$. On the other hand, $d\big{(}V(G[{\rm pr}_{G}(B_{1})]),{\rm pr}_{G}(B_{2})\big{)}\geq2$. This necessarily shows that $|B_{1}|=|B'|$ since $B$ is an EOP set of $G\circ H$ of maximum cardinality.

We now turn our attention to the set $B_{2}$. Recall by the definition that all edges in $B_{2}$ are of the form $(g,h)(g,h')$. Moreover, since $B$ is a $\rho_{e}^{o}(G\circ H)$-set, it follows that there are precisely $\rho_{e}^{o}(H)$ edges from $^gH$ for every $g\in {\rm pr}_{G}(B_{2})$. In fact, the equality $|B_{2}|=|{\rm pr}_{G}(B_{2})|\rho_{e}^{o}(H)$ holds. We then, by making use of the inequality (\ref{alpha}), have
\begin{equation*}
\begin{array}{lcl}
\rho_{e}^{o}(G\circ H)&=&|B_{1}|+|B_{2}|=|B'|+|{\rm pr}_{G}(B_{2})|\rho_{e}^{o}(H)\leq|{\rm pr}_{G}(B_{1})|\alpha(H)+(\alpha(G)-|{\rm pr}_{G}(B_{1})|)\rho_{e}^{o}(H)\\
&=&|{\rm pr}_{G}(B_{1})|\big{(}\alpha(H)-\rho_{e}^{o}(H)\big{)}+\alpha(G)\rho_{e}^{o}(H)\leq \rho_{e}^{o}(G)\big{(}\alpha(H)-\rho_{e}^{o}(H)\big{)}+\alpha(G)\rho_{e}^{o}(H)\\
&=&\rho_{e}^{o}(G)\alpha(H)+\rho_{e}^{o}(H)\big{(}\alpha(G)-\rho_{e}^{o}(G)\big{)}.  
\end{array}
\end{equation*}

That the upper bound is sharp may be seen by letting $H$ be an arbitrary graph, and $G$ be obtained from a star with center $c$ and leaves $v_{1},\ldots,v_{\ell}$, with $\ell\geq2$, by joining the endvertex $u$ of the path $P:uv$ to $v_{1}$. Let $I$ and $J$ be an $\alpha(H)$-set and a $\rho_{e}^{o}(H)$-set, respectively. It is easy to see that for an arbitrary vertex $h\in V(H)$, $B=\{(c,h)(v_{i},h')\mid 1\leq i\leq \ell\ \mbox{and}\ h'\in I\}\cup \{(v,x)(v,y)\mid xy\ \mbox{belongs to}\ J\}$ is an EOP set in $G\circ H$ attaining the upper bound.

We now prove the lower bound. Let $B$ and $I$ be a $\rho_{e}^{o}(G)$-set and an $\alpha(H)$-set, respectively. Let $s_{1},\ldots,s_{k}$ be the centers of the stars in $G[B]$. It is readily observed that 
\begin{center}
$B'=\{(s_{i},h)(g,h')\mid 1\leq i\leq k, g\ \mbox{is a leaf in $G[B]$ adjacent to $s_{i}$ and $h'\in I$}\}$
\end{center} 
is an EOP set in $G\circ H$ of cardinality $\rho_{e}^{o}(G)\alpha(H)$, in which $h$ is an arbitrary vertex of $H$. This leads to the lower bound.

Let $G$ be obtained from a star $S$ with center $s$ and the set of leaves $\{v_{1},\ldots,v_{\ell}\}$ by subdividing at most $0\leq t\leq \ell-1$ edges exactly once. Finally, we let $H$ and $I$ be any graph and an $\alpha(H)$-set, respectively. Then, for any vertex $h\in V(H)$, $\{(s,h)(u,h')\mid u\in N_{G}(s)\ \mbox{and}\ h'\in I\}$ is a $\rho_{e}^{o}(G\circ H)$-set of cardinality $\alpha(H)\ell=\rho_{e}^{o}(G)\alpha(H)$. This completes the proof.
\end{proof}

Note that $\nu_{I}(G)\alpha(H)$ for $\nu_{I}(G\circ H)$ is not a counterpart of the lower bound $\rho_{e}^{o}(G)\alpha(H)$ on $\rho_{e}^{o}(G\circ H)$ given in Theorem \ref{thm:lex}. In fact, it is readily seen that $\nu_{I}(P_{2}\circ P_{3n+1})=n<\lceil(3n+1)/2\rceil=\nu_{I}(P_{2})\alpha(P_{3n+1})$. 

In contrast with triangle-free graphs, a \textit{triangular graph} is a graph such that every edge lies in some triangle. These graphs are the underlying topic of several papers such as \cite{egt,pr,p}. In view of Theorems \ref{lex} and \ref{thm:lex}, we have $\im(G\circ K_{2})=\alpha(G)=\rho_{e}^{o}(K_{2}\circ G)$ for each graph $G$. Moreover, every edge of $G\circ K_{2}$ lies in a triangle for each graph $G$ with no isolated vertices. On the other hand, it is well known that {\sc Independent Set Problem} is NP-complete (see \cite{gj}). This leads to the following result concerning the decision versions of the induced matching number and the edge open packing number. 

\begin{corollary}
{\sc Induced Matching Problem} and {\sc Edge Open Packing Problem} are NP-complete even for triangular graphs.
\end{corollary}


\section{Direct product}\label{sec-dir}

The induced matching number of direct product of two graphs can be bounded from below in terms of the induced matching number of the factors.

\begin{theorem}\label{induced-direct}
For any graphs $G$ and $H$, $\nu_{I}(G\times H)\geq2\nu_{I}(G)\nu_{I}(H)$. This bound is sharp.
\end{theorem}
\begin{proof}
Let $M=\{g_{1}g_{1}',\ldots,g_{|M|}g_{|M|}'\}$ and $N=\{h_{1}h_{1}',\ldots,h_{|N|}h_{|N|}'\}$ be a $\nu_{I}(G)$-set and a $\nu_{I}(H)$-set, respectively. Consider the sets $V(M)=\{g_{1},g_{1}',\ldots,g_{|M|},g_{|M|}'\}$ and $V(N)=\{h_{1},h_{1}',\ldots,h_{|N|},h_{|N|}'\}$ of saturated vertices by $M$ and $N$, respectively. Since $M$ and $N$ are induced matchings, it follows that $|V(M)|=2|M|$, $|V(N)|=2|N|$, $N_{G}(g_{i})\cap V(M)=\{g_{i}'\}$ and $N_{G}(g_{i}')\cap V(M)=\{g_{i}\}$ for every $1\leq i\leq|M|$, and $N_{H}(h_{j})\cap V(N)=\{h_{j}'\}$ and $N_{H}(h_{j}')\cap V(N)=\{h_{j}\}$ for every $1\leq j\leq|N|$. We let 
\begin{center} 
$Q_{i}=\big{\{}(g_{i},h_{j})(g_{i}',h_{j}')\in E(G\times H)\mid 1\leq j\leq|N|\big{\}}$
\end{center}
for each $1\leq i\leq|M|$, and set $Q=\bigcup_{i=1}^{|M|}Q_{i}$. Since $M$ and $N$ are induced matchings in $G$ and $H$ respectively, it is clear that no two edges of $Q$ are adjacent in $G \times H$.

Suppose that two edges $(g_{i},h_{j})(g_{i}',h_{j}'),(g_{s},h_{t})(g_{s}',h_{t}')\in Q$, for some $1\leq i,s\leq|M|$ and $1\leq j,t\leq|N|$, have a common edge $e$. Without loss of generality, we may assume that $e=(g_{i},h_{j})(g_s',h_t')$. By the chosen notation and since $N_{G}(g_{i})\cap V(M)=\{g_{i}'\}$, this implies that $i=s$. This in turn leads to $h_{j}h_{t}'\in E(H)$, while $j\neq t$. This is a contradiction. In fact, we have proved that no two edges in $Q$ have a common edge. The above discussion guarantees that $Q$ is an induced matching in $G\times H$. 

We now define 
\begin{center} 
$Q_{i}'=\big{\{}(g_{i}',h_{j})(g_{i},h_{j}')\in E(G\times H)\mid 1\leq j\leq|N|\big{\}}$
\end{center}
for every $1\leq i\leq|M|$. We observe, by a similar fashion, that $Q'=\bigcup_{i=1}^{|M|}Q_{i}'$ is also an induced matching in $G\times H$. Clearly, $Q\cap Q'=\emptyset$ and $|Q|=|Q'|$. Moreover, no two edges $e\in Q$ and $e'\in Q'$ have a common endvertex by the observations given in the first paragraph of the proof. Therefore, $Q\cup Q'$ is a matching in $G\times H$. Let $(g_{i},h_{j})(g_{i}',h_{j}')\in Q$ and $(g_{s}',h_{t})(g_{s},h_{t}')\in Q'$ have a common edge $e$. We may assume that $(g_{i},h_{j})$ is incident with $e$. Since $g_{i}g_{s}\notin E(G)$, we deduce that $e=(g_{i},h_{j})(g_{s}',h_{t})\in E(G\times H)$. This in turn implies that $h_{j}h_{t}\in E(H)$, contrary to $N$ being an induced matching in $H$. Consequently, $Q\cup Q'$ is an induced matching in $G\times H$. Thus, $\nu_{I}(G\times H)\geq|Q\cup Q'|=2|Q|=2\nu_{I}(G)\nu_{I}(H)$.

The bound is sharp for several infinite families of graphs. For instance, we verify the sharpness for $P_{3m}\times K_{n}$ for all integers $m\geq1$ and $n\geq3$. Note that $P_{3m}\times K_{n}$ is isomorphic to the graph $F$ constructed as follows. Let $V(F)=V_{1}\cup \ldots \cup V_{3m}$, in which $V_{i}=\{v_{i1},\ldots,v_{in}\}$ for each $1\leq i\leq3m$, with edges with one endvertex in $V_{i}$ and the other in $V_{i+1}$ such that the resulting induced subgraph $F[V_{i}\cup V_{i+1}]$ is obtained from $K_{n,n}$ by removing the perfect matching $\{v_{i1}v_{(i+1)1},\ldots,v_{in}v_{(i+1)n}\}$ for every $1\leq i\leq3m-1$. It is readily seen, by the definition of an induced matching and the structure of $F$, that an induced matching in $F$ has at most two edges from $F[V_{3i-2}\cup V_{3i-1}\cup V_{3i}]$ for each $1\leq i\leq m$. This implies that $\nu_{I}(F)\leq2m$. On the other hand, it is clear that $\nu_{I}(P_{3m})=m$. Therefore, $\nu_{I}(F)\geq2\nu_{I}(P_{3m})\nu_{I}(K_{n})=2m$. This completes the proof. 
\end{proof}

Note that the inequality given in Theorem \ref{induced-direct} cannot be generalized to EOP numbers. That is, $\rho_{e}^{o}(G\times H)\geq2\rho_{e}^{o}(G)\rho_{e}^{o}(H)$ is not true as it stands. In fact, $P_{3}\times P_{12n-5}$ for all integers $n\geq1$ serve as an infinite family of counterexamples. To see this, it can be readily checked that $\rho_{e}^{o}(P_{3}\times P_{12n-5})=24n-10$ while $2\rho_{e}^{o}(P_{3})\rho_{e}^{o}(P_{12n-5})=24n-8$ (see Figure \ref{Figure2} for $n=2$).

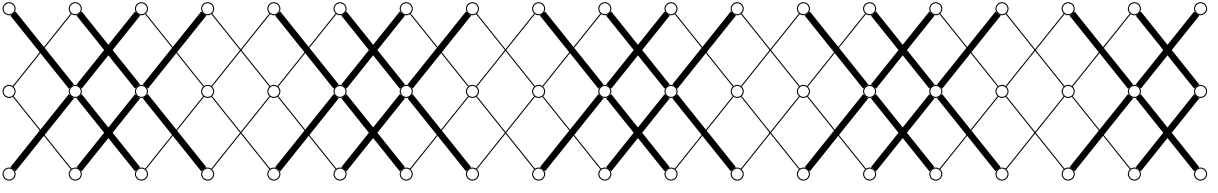
\begin{figure}[ht]
\centering
\begin{tikzpicture}[scale=0.44, transform shape]
\node [draw, shape=circle] (x1) at (0,0) {};
\node [draw, shape=circle] (x2) at (2,0) {};
\node [draw, shape=circle] (x3) at (4,0) {};
\node [draw, shape=circle] (x4) at (6,0) {};
\node [draw, shape=circle] (x5) at (8,0) {};
\node [draw, shape=circle] (x6) at (10,0) {};
\node [draw, shape=circle] (x7) at (12,0) {};
\node [draw, shape=circle] (x8) at (14,0) {};
\node [draw, shape=circle] (x9) at (16,0) {};
\node [draw, shape=circle] (x10) at (18,0) {};
\node [draw, shape=circle] (x11) at (20,0) {};
\node [draw, shape=circle] (x12) at (22,0) {};
\node [draw, shape=circle] (x13) at (24,0) {};
\node [draw, shape=circle] (x14) at (26,0) {};
\node [draw, shape=circle] (x15) at (28,0) {};
\node [draw, shape=circle] (x16) at (30,0) {};
\node [draw, shape=circle] (x17) at (32,0) {};
\node [draw, shape=circle] (x18) at (34,0) {};
\node [draw, shape=circle] (x19) at (36,0) {};


\node [draw, shape=circle] (y1) at (0,-2.5) {};
\node [draw, shape=circle] (y2) at (2,-2.5) {};
\node [draw, shape=circle] (y3) at (4,-2.5) {};
\node [draw, shape=circle] (y4) at (6,-2.5) {};
\node [draw, shape=circle] (y5) at (8,-2.5) {};
\node [draw, shape=circle] (y6) at (10,-2.5) {};
\node [draw, shape=circle] (y7) at (12,-2.5) {};
\node [draw, shape=circle] (y8) at (14,-2.5) {};
\node [draw, shape=circle] (y9) at (16,-2.5) {};
\node [draw, shape=circle] (y10) at (18,-2.5) {};
\node [draw, shape=circle] (y11) at (20,-2.5) {};
\node [draw, shape=circle] (y12) at (22,-2.5) {};
\node [draw, shape=circle] (y13) at (24,-2.5) {};
\node [draw, shape=circle] (y14) at (26,-2.5) {};
\node [draw, shape=circle] (y15) at (28,-2.5) {};
\node [draw, shape=circle] (y16) at (30,-2.5) {};
\node [draw, shape=circle] (y17) at (32,-2.5) {};
\node [draw, shape=circle] (y18) at (34,-2.5) {};
\node [draw, shape=circle] (y19) at (36,-2.5) {};


\node [draw, shape=circle] (z1) at (0,-5) {};
\node [draw, shape=circle] (z2) at (2,-5) {};
\node [draw, shape=circle] (z3) at (4,-5) {};
\node [draw, shape=circle] (z4) at (6,-5) {};
\node [draw, shape=circle] (z5) at (8,-5) {};
\node [draw, shape=circle] (z6) at (10,-5) {};
\node [draw, shape=circle] (z7) at (12,-5) {};
\node [draw, shape=circle] (z8) at (14,-5) {};
\node [draw, shape=circle] (z9) at (16,-5) {};
\node [draw, shape=circle] (z10) at (18,-5) {};
\node [draw, shape=circle] (z11) at (20,-5) {};
\node [draw, shape=circle] (z12) at (22,-5) {};
\node [draw, shape=circle] (z13) at (24,-5) {};
\node [draw, shape=circle] (z14) at (26,-5) {};
\node [draw, shape=circle] (z15) at (28,-5) {};
\node [draw, shape=circle] (z16) at (30,-5) {};
\node [draw, shape=circle] (z17) at (32,-5) {};
\node [draw, shape=circle] (z18) at (34,-5) {};
\node [draw, shape=circle] (z19) at (36,-5) {};


\draw (x1)--(y2)--(x3)--(y4)--(x5)--(y6)--(x7)--(y8)--(x9)--(y10)--(x11)--(y12)--(x13)--(y14)--(x15)--(y16)--(x17)--(y18)--(x19);
\draw (y1)--(x2)--(y3)--(x4)--(y5)--(x6)--(y7)--(x8)--(y9)--(x10)--(y11)--(x12)--(y13)--(x14)--(y15)--(x16)--(y17)--(x18)--(y19);
\draw (y1)--(z2)--(y3)--(z4)--(y5)--(z6)--(y7)--(z8)--(y9)--(z10)--(y11)--(z12)--(y13)--(z14)--(y15)--(z16)--(y17)--(z18)--(y19);
\draw (z1)--(y2)--(z3)--(y4)--(z5)--(y6)--(z7)--(y8)--(z9)--(y10)--(z11)--(y12)--(z13)--(y14)--(z15)--(y16)--(z17)--(y18)--(z19);


\draw[line width=2.5pt] (x1)--(y2)--(x3);
\draw[line width=2.5pt] (x5)--(y6)--(x7);
\draw[line width=2.5pt] (x9)--(y10)--(x11);
\draw[line width=2.5pt] (x13)--(y14)--(x15);
\draw[line width=2.5pt] (x17)--(y18)--(x19);

\draw[line width=2.5pt] (z1)--(y2)--(z3);
\draw[line width=2.5pt] (z5)--(y6)--(z7);
\draw[line width=2.5pt] (z9)--(y10)--(z11);
\draw[line width=2.5pt] (z13)--(y14)--(z15);
\draw[line width=2.5pt] (z17)--(y18)--(z19);


\draw[line width=2.5pt] (x2)--(y3)--(x4);
\draw[line width=2.5pt] (x6)--(y7)--(x8);
\draw[line width=2.5pt] (x10)--(y11)--(x12);
\draw[line width=2.5pt] (x14)--(y15)--(x16);
\draw[line width=2.5pt] (x18)--(y19);

\draw[line width=2.5pt] (z2)--(y3)--(z4);
\draw[line width=2.5pt] (z6)--(y7)--(z8);
\draw[line width=2.5pt] (z10)--(y11)--(z12);
\draw[line width=2.5pt] (z14)--(y15)--(z16);
\draw[line width=2.5pt] (z18)--(y19);

\end{tikzpicture}
\caption{{\small The graph $P_{3}\times P_{19}$. The bold edges form an optimal EOP set.}}\label{Figure2}
\end{figure}

In view of the above remark, we employ a different approach so as to bound the EOP number from below in the case of direct product graphs.

\begin{theorem}\label{Direct}
For any graphs $G$ and $H$, 
$$\rho_{e}^{o}(G\times H)\geq \max\Big{\{}\rho_{e}^{o}(G)\delta(H)\rho^{o}(H),\rho_{e}^{o}(H)\delta(G)\rho^{o}(G)\Big{\}}.$$ 
Moreover, this bound is sharp. 
\end{theorem}
\begin{proof}
If at least one of $G$ and $H$ is an edgeless graph, then so is $G\times H$, and the lower bound trivially holds with equality. So, we assume that both $G$ and $H$ have edges.

Let $A$ be a $\rho_{e}^{o}(G)$-set. Note, by the definition, that the subgraph $F$ of $G$ induced by $A$ is a disjoint union of stars. Let $K$ be the vertex subset of $G$ containing the centers of the stars in $F$ (if a star in $F$ is isomorphic to $P_{2}$, we choose only one of its endvertices). Consider an arbitrary edge $gg'\in A$, in which $g\in K$. Let $B=\{h_{1},\ldots,h_{\rho^{o}(H)}\}$ be a $\rho^{o}(H)$-set. For each $1\leq i\leq \rho^{o}(H)$, by the adjacency rule of direct product graphs, $(g,h_{i})$ is adjacent to all vertices $(g',h)$ with $h\in N_{H}(h_{i})$. We now set 
$$P=\bigcup_{g\in K}\bigcup_{i=1}^{\rho^{o}(H)}\Big{\{}(g,h_{i})(g',h)\mid gg'\in A\ \mbox{and}\ h\in N_{H}(h_{i})\Big{\}}.$$

Suppose that there exist two edges $e=(g_{1},h_{i})(g_{1}',h)$ and $f=(g_{2},h_{j})(g_{2}',h')$ in $P$ having a common edge in $G\times H$. We distinguish two cases depending on $g_{1}$ and $g_{2}$.\vspace{0.75mm}\\
\textit{Case 1}. Let $g_{1}=g_{2}=g$ for some $g\in K$. Suppose first that $g_{1}'=g_{2}'$. We need to consider two more possibilities.\vspace{0.75mm}\\
\textit{Subcase 1.1}. $h_{i}=h_{j}$. Then, both $(g_{1}',h)$ and $(g_{2}',h')$ are adjacent to $(g,h_{i})$. In such a situation,  the common edge joins $(g_{1}',h)$ to $(g_{2}',h')$. This is a contradiction since $g_{1}'=g_{2}'$.\vspace{0.75mm}\\
\textit{Subcase 1.2}. $h_{i}\neq h_{j}$. If $h=h'$, then $h\in N_{H}(h_{i})\cap N_{H}(h_{j})$. This is impossible due to the fact that $h_{i},h_{j}\in B$ and that $B$ is an open packing in $H$. Therefore, $h\neq h'$. Note that the edges $(g,h_{i})(g,h_{j})$ and $(g_{1}',h)(g_{2}',h')$ do not exist in $G\times H$ by the adjacency rule of $G\times H$. So, we may assume without loss of generality that $(g_{1}',h)(g,h_{j})\in E(G\times H)$ is the common edge. In particular, this implies that $h\in N_{H}(h_{i})\cap N_{H}(h_{j})$. This is a contradiction as $B$ is an open packing in $H$ and $h_{i},h_{j}\in B$.\vspace{0.75mm}

Suppose now that $g_{1}'\neq g_{2}'$. We consider two possibilities depending on $h_{i}$ and $h_{j}$.\vspace{0.75mm}\\
\textit{Subcase 1.3}. $h_{i}=h_{j}$. In such a case, $(g_{1}',h)$ must be adjacent to $(g_{2}',h')$. In particular, we have $g_{1}'g_{2}'\in E(G)$. This implies the existence of the triangle $gg_{1}'g_{2}'g$ in $G$, in which the edges $gg_{1}'$ and $gg_{2}'$ belong to $A$. This is a contradiction because $A$ is an EOP set in $G$.\vspace{0.75mm}\\
\textit{Subcase 1.4}. $h_{i}\neq h_{j}$. By the adjacency rule of $G\times H$, we infer that $(g_{1}',h)(g_{2}',h')\in E(G\times H)$, or, we can assume without loss of generality that $(g,h_{i})(g_{2}',h')\in E(G\times H)$. Note that $(g_{1}',h)(g_{2}',h')\in E(G\times H)$ does not happen, otherwise we have the triangle $gg_{1}'g_{2}'g$ in $G$ with $gg_{1}',gg_{2}'\in A$. On the other hand, $(g,h_{i})(g_{2}',h')\in E(G\times H)$ implies that $h'\in N_{H}(h_{i})\cap N_{H}(h_{j})$, contradicting the fact that $B$ is an open packing in $H$ and $h_{i},h_{j}\in B$.\vspace{0.75mm}\\
\textit{Case 2.} Let $g_{1}\neq g_{2}$. By our choices of the vertices in $K$, $g_{1}$ and $g_{2}$ are the centers of two stars $K_{1,a}$ and $K_{1,b}$ in the subgraph of $G$ induced by $A$, respectively. Since $A$ is an EOP in $G$, it follows that $d_{G}\big{(}V(K_{1,a}),V(K_{1,b})\big{)}\geq2$. This in particular guarantees that $g_{1}g_{2},g_{1}g_{2}',g_{1}'g_{2},g_{1}'g_{2}'\notin E(G)$. Therefore, $(g_{1},h_{i})(g_{1}',h)$ and $(g_{2},h_{j})(g_{2}',h')$ do not have any common edge.\vspace{0.75mm}

All in all, we have proved that every two distinct edges in $P$ have no common edges. Hence, $P$ is an EOP set in $G\times H$. We now consider an arbitrary vertex $g\in K$ and let $g$ be the center of the star $K_{1,a}$ in the induced subgraph of $G$ by $A$. Corresponding to each edge of $K_{1,a}$ (as a subgraph of $G$), $(g,h_{i})$ is incident with $\deg_{H}(h_{i})$ edges in $P$. Consequently, 
$$\rho_{e}^{o}(G\times H)\geq|P|=\rho_{e}^{o}(G)\sum_{i=1}^{\rho^{o}(H)}\deg_{H}(h_{i})\geq \rho_{e}^{o}(G)\delta(H)\rho^{o}(H).$$
Interchanging the roles of $G$ and $H$ yields the desired lower bound.

For the sharpness of the bound, consider the direct product $K_{m}\times K_{n}$ with $m\ge n\geq3$. Let $B$ be a $\rho_{e}^{o}(K_{m}\times K_{n})$-set. Fix an edge $e=(g,h)(g',h')\in B$. Let $Q_{1}=[\{g\}\times V(H),\{g'\}\times V(H)]$ and $Q_{2}=[V(G)\times \{h\},V(G)\times \{h'\}]$. Let $f$ be an edge of $K_{m}\times K_{n}$ that does not belong to $Q_1\cup Q_2$. If $f$ does not have a common endvertex with $e$, then clearly $f$ has a common edge with $e$, and so $f\notin B$. On the other hand, if $f$ has a common endvertex with $e$, then $e$ and $f$ lie on a triangle in $K_{m}\times K_{n}$, thus they again have a common edge. Both observations together imply that $B\subseteq Q_1\cup Q_2$. 

Now, consider the edges in $Q_1$ (and $Q_2$). If an edge $f\in Q_1$ does not have a common endvertex with $e$, then $f=(g,h_1)(g',h_2)$, where $h_1\ne h$ and $h_2\ne h'$. 
Suppose that $f\ne (g,h')(g',h)$. Then, it is easy to see that $e$ and $f$ have a common edge, and so $f\notin B$. By symmetry, we derive that an edge $f$ in $Q_2\setminus \{(g,h')(g',h)\}$, which does not have a common endvertex with $e$, does not belong to $B$. We distinguish two cases. First, suppose $(g,h')(g',h)\in B$. Then, any edge in $(Q_1\cup Q_2)\setminus \{(g,h')(g',h),e\}$ has a common edge with at least one edge in $\{(g,h')(g',h),e\}$, which implies $|B|=2$. This is possible only when $m=n=3$, in which case we indeed have $\rho_{e}^{o}(K_{3}\times K_{3})=2$. Second, let $(g,h')(g',h)\notin B$. We have already established that every edge $f$ in $B$ must be in $Q_1\cup Q_2$, and must have a common endvertex with $e$. Suppose that $\{f_1,f_2\}\subseteq B\setminus\{e\}$, where $f_1$ has $(g,h)$ as an endvertex, while $f_2$ has $(g',h')$ as an endvertex. Then, $e$ is a common edge of $f_1$ and $f_2$, a contradiction with $B$ being an EOP set. Thus, we may assume without loss of generality that all edges in $B$ have $(g,h)$ as an endvertex. Taking any two such edges, where $f_1\in Q_1\setminus\{e\}$ and $f_2\in Q_2\setminus\{e\}$, we note that $f_1$ and $f_2$ cannot both be in $B$ since they lie in a triangle. Therefore, $B$  consists of either all edges in $Q_1$ or all edges in $Q_2$ whose common endvertex is $(g,h)$. Since $m\ge n$, we may assume that the latter possibility appears, and so $|B|=m-1=\rho_{e}^{o}(K_{m}\times K_{n})=\max\Big{\{}\rho_{e}^{o}(K_m)\delta(K_n)\rho^{o}(K_n),\rho_{e}^{o}(K_n)\delta(K_m)\rho^{o}(K_m)\Big{\}}$. 
\end{proof}

As shown in the proof of Theorem \ref{Direct}, the lower bound is sharp for some families of graphs with large minimum degree. However, there are some infinite families of graphs with small minimum degree attaining the lower bound. Jha and Klav\v{z}ar (\cite{JK}) proved that $\alpha(C_{m}\times C_{n})=mn/2$ when both $m$ and $n$ are even. With this in mind and the fact that $\rho_{e}^{o}(G)\leq \alpha(G)$ for all graphs $G$, we deduce that $8mn=\rho_{e}^{o}(C_{4m})\delta(C_{4n})\rho^{o}(C_{4n})\leq \rho_{e}^{o}(C_{4m}\times C_{4n})\leq8mn$.

We close this section by remarking that the $\nu_{I}$ version of the inequality given in Theorem \ref{Direct} does not hold. In fact, an argument similar to that for $K_{m}\times K_{n}$ in the proof of the theorem shows that $\im(K_{m}\times K_{n})=2$ for $m\geq n\geq4$, while $\max\big{\{}\im(K_m)\delta(K_n)\rho^{o}(K_n),\im(K_n)\delta(K_m)\rho^{o}(K_m)\big{\}}=m-1\geq3$.


\section{Cartesian and strong products}\label{sec-Cp}

Many graph properties are hereditary in the sense that (spanning) subgraphs inherit the property that is satisfied by the considered graph. In particular, this holds for a number of properties that arise from graph invariants such as the chromatic number, the chromatic index, the independence number, the (open) packing number, etc. However,  $\rho_{e}^{o}(H)$ and $\rho_{e}^{o}(G)$ are incomparable in general, where $H$ is a spanning subgraph of $G$. Moreover, $\rho_{e}^{o}(G)-\rho_{e}^{o}(H)$ and $\rho_{e}^{o}(G)-\rho_{e}^{o}(H)$ can be arbitrarily large. 

\begin{obs}\label{obs1}
For every integer $r\geq1$, there exist a connected graph $G$ with a connected spanning subgraph $H$ of $G$ such that $\rho_{e}^{o}(G)-\rho_{e}^{o}(H)=r$ \emph{(}respectively, $\rho_{e}^{o}(H)-\rho_{e}^{o}(G)=r$\emph{)}.
\end{obs}

To see this, first consider the graph $G$ depicted in Figure \ref{Figure1} and let $H=G-\{x_{1}y_{1},\ldots,x_{2r+1}y_{2r+1}\}$. It is then easy to verify that $\big{(}\rho_{e}^{o}(G),\rho_{e}^{o}(H)\big{)}=(4r+2,3r+2)$. On the other hand, assuming that $G'=K_{2r+3}$ and $H'$ is a spanning cycle in $G'$, we have $\big{(}\rho_{e}^{o}(G'),\rho_{e}^{o}(H')\big{)}=(1,r+1)$.

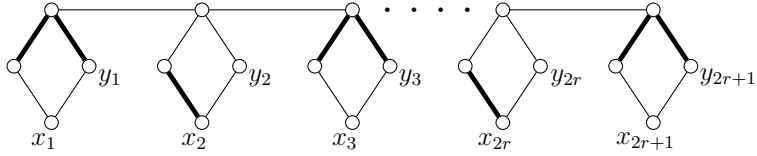
\begin{figure}[ht]
 \centering
\begin{tikzpicture}[scale=0.5, transform shape]
\node [draw, shape=circle] (x1) at (0,0) {};
\node [draw, shape=circle] (x2) at (4,0) {};
\node [draw, shape=circle] (x3) at (8,0) {};
\node [scale=1.7] at (10,0) {\huge $.\ .\ .\ .$};
\node [draw, shape=circle] (x4) at (12,0) {};
\node [draw, shape=circle] (x5) at (16,0) {};

\node [scale=1.7] at (-0.2,-3.5) {\large $x_{1}$};
\node [scale=1.7] at (3.8,-3.5) {\large $x_{2}$};
\node [scale=1.7] at (7.8,-3.5) {\large $x_{3}$};
\node [scale=1.7] at (11.8,-3.5) {\large $x_{2r}$};
\node [scale=1.7] at (15.8,-3.5) {\large $x_{2r+1}$};

\node [scale=1.7] at (1.55,-1.8) {\large $y_{1}$};
\node [scale=1.7] at (5.55,-1.8) {\large $y_{2}$};
\node [scale=1.7] at (9.55,-1.8) {\large $y_{3}$};
\node [scale=1.7] at (13.7,-1.8) {\large $y_{2r}$};
\node [scale=1.7] at (18,-1.8) {\large $y_{2r+1}$};

\draw (x1)--(x2)--(x3);
\draw (x4)--(x5);

\node [draw, shape=circle] (x11) at (-1,-1.5) {};
\node [draw, shape=circle] (x12) at (1,-1.5) {};
\node [draw, shape=circle] (x13) at (0,-3) {};

\draw (x11)--(x13)--(x12);
\draw[line width=1.8pt] (x11)--(x1)--(x12);

\node [draw, shape=circle] (x21) at (3,-1.5) {};
\node [draw, shape=circle] (x22) at (5,-1.5) {};
\node [draw, shape=circle] (x23) at (4,-3) {};

\draw (x21)--(x2)--(x22)--(x23);
\draw[line width=1.8pt] (x21)--(x23);

\node [draw, shape=circle] (x31) at (7,-1.5) {};
\node [draw, shape=circle] (x32) at (9,-1.5) {};
\node [draw, shape=circle] (x33) at (8,-3) {};

\draw (x31)--(x33)--(x32);
\draw[line width=1.8pt] (x31)--(x3)--(x32);

\node [draw, shape=circle] (x41) at (11,-1.5) {};
\node [draw, shape=circle] (x42) at (13,-1.5) {};
\node [draw, shape=circle] (x43) at (12,-3) {};

\draw (x41)--(x4)--(x42)--(x43);
\draw[line width=1.8pt] (x41)--(x43);

\node [draw, shape=circle] (x51) at (15,-1.5) {};
\node [draw, shape=circle] (x52) at (17,-1.5) {};
\node [draw, shape=circle] (x53) at (16,-3) {};

\draw (x51)--(x53)--(x52);
\draw[line width=1.8pt] (x51)--(x5)--(x52);

\end{tikzpicture}
\caption{{\small The graph $G$. The bold edges form a $\rho_{e}^{o}(H)$-set of cardinality $3r+2$.}}\label{Figure1}
\end{figure}

In view of Observation \ref{obs1}, the inequality stated in Proposition \ref{prop1} cannot be derived from the fact that $G\boxtimes H$ and $G\square H$ are spanning subgraphs of $G\circ H$ for all graphs $G$ and $H$.

\begin{proposition}\label{prop1}
For any graphs $G$ and $H$, $\rho_{e}^{o}(G\circ H)\leq min\{\rho_{e}^{o}(G\boxtimes H),\rho_{e}^{o}(G\square H)\}$. 
\end{proposition}
\begin{proof}
We adopt the notations given in the proof of Theorem \ref{thm:lex}. Consider any edge $e=gg'\in \bigcup_{i=1}^{k}E(S_{i})$. Then, $e=p_{G}\big{(}(g,h)(g',h')\big{)}$ for some $(g,h)(g',h')\in B_{1}$. 

Let $e$ be the projection of two edges $(g,h_{1})(g',h_{1}'),(g,h_{2})(g',h_{2}')\in B_{1}$ which have no common endvertex. Then, by the adjacency rule of lexicographic product, $(g',h_{1}')(g,h_{2})$ is a common edge of $(g,h_{1})(g',h_{1}')$ and $(g,h_{2})(g',h_{2}')$, a contradiction. This shows that every two edges in $B_{1}$ with the projection $e$ to $G$ have a common endvertex. On the other hand, if there exists an edge of the form $(g',h_{1}')(g',h_{2}')\in B$, then either $(g,h)(g',h_{1}')$ or $(g,h)(g',h_{2}')$ is a common edge of $(g,h)(g',h')$ and $(g',h_{1}')(g',h_{2}')$, which is impossible. Moreover, there is no edge of the form $(g,h_{1})(g,h_{2})\in B$ by a similar fashion.

Suppose that $e$ is the projection of $k\geq \alpha(H)+1$ edges $(g,h)(g',h_{1}'),\ldots,(g,h)(g',h_{k}')\in B_{1}$ to $G$. Hence, $|\{h_{1}',\ldots,h_{k}'\}|=k\geq \alpha(H)+1$. Therefore, $h_{i}'h_{j}'\in E(H)$ for some $1\leq i\neq j\leq k$. This shows that $(g',h_{i}')(g',h_{j}')$ is the common edge of $(g,h)(g',h_{i}')$ and $(g,h)(g',h_{j}')$, contradicting the fact that $B$ is an EOP set of $G\circ H$. In fact, we have proved that every edge of $\bigcup_{i=1}^{k}E(S_{i})$ is the projection of at most $\alpha(H)$ edges of $B_{1}$. Consequently, $|B_{1}|\leq \sum_{i=1}^{k}|E(S_{i})|\alpha(H)$. 

On the other hand, let $I$ be an $\alpha(H)$-set and fix a vertex $h^{*}\in V(H)$. We observe that
\begin{center} 
$B'=\bigcup_{i=1}^{k}\bigcup_{g\in L(S_{i})}\bigcup_{h\in I}\{(s_{i},h^{*})(g,h)\}$
\end{center}
is an EOP set of $G\circ H$ of cardinality $\sum_{i=1}^{k}|E(S_{i})|\alpha(H)$, in which $L(S_{i})$ and $s_i$ are the set of leaves of $S_{i}$ and the center of $S_i$, respectively, and every edge is of the form $(g_1,h_1)(g_2,h_2)$ with $g_{1}\neq g_{2}$. In view of this and since $B$ is an EOP set of $G\circ H$ of maximum size, we may assume that $B_{1}=B'$.

We now construct a desired EOP set of both $G\square H$ and $G\boxtimes H$ by using the EOP set $B$ of $G\circ H$. Let $V(S_{i})=\{s_{i},g_{i1},\ldots,g_{im_{i}}\}$ for each $i\in[k]$. By the structure of Cartesian and strong product graphs, we deduce that 
\begin{center}
$(\bigcup_{i=1}^{k}\bigcup_{h\in I}\{(s_{i},h)(g_{i1},h),\ldots,(s_{i},h)(g_{im_{i}},h)\})\bigcup B_{2}$ 
\end{center}
is an EOP set in both $G\square H$ and $G\boxtimes H$ of cardinality $|B|$. This leads to $\rho_{e}^{o}(G\circ H)=|B|\leq \rho_{e}^{o}(G\boxtimes H),\rho_{e}^{o}(G\square H)$ and completes the proof.  
\end{proof}

Let $I$ and $B$ be an $\alpha(G)$-set and a $\rho_{e}^{o}(H)$-set, respectively. By the adjacency rule of $G\circ H$, we observe that $Q=\bigcup_{g\in I}\{(g,h)(g,h')\mid hh'\in B\}$ is an EOP set of $G\circ H$. Hence, $\rho_{e}^{o}(G\circ H)\geq|Q|=\alpha(G)\rho_{e}^{o}(H)$. With this in mind, we have the following bounds in the cases of Cartesian and strong product graphs due to Proposition \ref{prop1} and Theorem \ref{thm:lex}.

\begin{corollary}\label{cart-str}
For any graphs $G$ and $H$, \vspace{1mm}\\
$(i)$ $\rho_{e}^{o}(G\square H)\ge \max\{\rho_{e}^{o}(G)\alpha(H),\alpha(G)\rho_{e}^{o}(H)\}$, and\vspace{1mm}\\
$(ii)$ $\rho_{e}^{o}(G\boxtimes H)\ge \max\{\rho_{e}^{o}(G)\alpha(H),\roeo(H)\alpha(G)\}$.\vspace{1mm}\\
These bounds are sharp.
\end{corollary}
\begin{proof}
We only need to prove the sharpness. Consider the Cartesian product of two stars, notably $G\cong K_{1,r}$ and $H\cong K_{1,s}$, where $s\ge r\ge 2$. Let $V(G)=\{x,x_1,\ldots,x_r\}$ and $V(H)=\{y,y_1,\ldots,y_s\}$, where $x$ and $y$ are the centers of the stars. As the case when $r=s=2$ is clear, we may assume $s\ge3$. Next, it is easy to see that $P=\{(x_i,y)(x_i,y_j)\mid 1\leq i\leq r\ \mbox{and}\ 1\leq j\leq s\}$ is an EOP set in $G\square H$ of cardinality $rs$. Note that all vertices of the form $(x_i,y_j)$ are leaves in $(G\square H)[P]$.

Let $Q$ be a $\rho_{e}^{o}(G\square H)$-set. Suppose first that $(x,y)$ is a vertex of $(G\square H)[Q]$. We may assume, without loss of generality, that $(x,y)(x,y_1)$ is an edge in $Q$. By the structure and since $Q$ is an EOP set of $G\square H$, neither the edges of fibers $^{x_{i}}\!H$, for $i\in[r]$, nor the edges of fibers $G^{y_{j}}$, for $j\in[s]\setminus\{1\}$, exist in $Q$. Moreover, if an edge of $G^{y}$ (resp.\ $G^{y_{1}}$) is in $Q$, then no edge of $G^{y_{1}}$ (resp.\ $G^{y}$) is in $Q$. Therefore, $Q\subseteq E(^{x}\!H)\cup E(G^{y})$ or $Q\subseteq E(^{x}\!H)\cup E(G^{y_{1}})$. So, $\rho_{e}^{o}(G\square H)\leq r+s<rs$, a contradiction.

Since $(x,y)\notin V\big{(}(G\square H)[Q]\big{)}$, all edges in $Q$ must have an endvertex of the form $(x_i,y_j)$. Suppose that there exists a vertex of the form $(x_i,y_j)$, which is incident with two edges from $Q$ (note that $\deg_{G\cp H}\big{(}(x_i,y_j)\big{)}=2$, hence it can be adjacent with at most two edges from $Q$), and let $S$ be the set of all such vertices. Without loss of generality, assume that $(x_1,y_1)\in S$, and so it is the center of a star in $(G\square H)[Q]$ with two leaves \big{(}namely, $(x_1,y)$ and $(x,y_1)$\big{)}. Note that none of the vertices in $V(^{x_1}\!H)\setminus\{(x_1,y),(x_1,y_1)\}$ can belong to $V\big{(}(G\square H)[Q]\big{)}$, and also none of the vertices in $V(G^{y_1})\setminus \{(x,y_1),(x_1,y_1)\}$ can be in $V\big{(}(G\square H)[Q]\big{)}$, because $Q$ is an EOP set. Therefore, we may assume without loss of generality that $S=\{(x_1,y_1),\ldots,(x_k,y_k)\}$, where $1\le k\le r$. By the previous observation, we infer that
\begin{center}
$|Q|=2k+(r-k)(s-k)=rs-k(s+r-k-2)<rs$.
\end{center}
This is a contradiction with $Q$ being a $\roeo(G\cp H)$-set. Therefore, every vertex of the form $(x_i,y_j)$ is incident with at most one edge from $Q$. With this in mind and since each edge in $Q$ has an endvertex of the form $(x_i,y_j)$, we derive that $|Q|\leq|\{(x_i,y_j)\mid i\in[r]\ \mbox{and}\ j\in[s]\}|=rs$. On the other hand, by the construction of the EOP set $P$ presented earlier, we infer $|Q|=rs$, which shows that the bound $(i)$ is sharp.

For any graph $G$ and integer $n\geq3$, $\rho_{e}^{o}(G\boxtimes K_n)\ge \max\{\rho_{e}^{o}(G)\cdot 1,1\cdot\alpha(G)\}=\alpha(G)$ by $(ii)$. Moreover, $G\boxtimes K_n\cong G\circ K_n$. This, in view of Theorem \ref{thm:lex} leads to $\rho_{e}^{o}(G\boxtimes K_n)\leq \alpha(G)$. Therefore, $(ii)$ is sharp regardless of our choice of $G$.
\end{proof}

By the sharpness part of the proof of Corollary\ref{cart-str} $(ii)$ and the fact that $G\circ K_n\cong G\boxtimes K_n$ we infer that $\rho_{e}^{o}(G\circ K_n)=\alpha(G)$, where $G$ is an arbitrary graph and $n\ge 3$. This gives yet another large family of pairs of graphs for which the upper bound in Theorem \ref{thm:lex} is attained.

Note that the results in Proposition \ref{prop1} and Corollary \ref{cart-str} have natural induced matching counterparts as follows.

\begin{proposition}\label{prop2}
The following statements hold for any graphs $G$ and $H$:\vspace{1mm}\\
$(i)$ $\nu_{I}(G\circ H)\leq min\{\nu_{I}(G\boxtimes H),\nu_{I}(G\square H)\}$,\vspace{1mm}\\
$(ii)$ $\nu_{I}(G\square H)\ge \max\{\nu_{I}(G)\alpha(H),\alpha(G)\nu_{I}(H)\}$, and\vspace{1mm}\\
$(iii)$ $\nu_{I}(G\boxtimes H)\ge \max\{\nu_{I}(G)\alpha(H),\nu_{I}(H)\alpha(G)\}$.
\end{proposition}

\subsection{Hypercubes}

The {\em $n$-dimensional hypercube}, or {\em $n$-cube}, $Q_n$ can be recursively defined as $Q_1=K_2$ and $Q_n=Q_{n-1}\cp K_2$ when $n\ge 2$. Clearly, $n$-cubes are $n$-regular bipartite graphs.

The exact value of induced matching number in hypercubes is known. In fact, $\nu_I(Q_n)=2^{n-2}$ for all $n\ge1$ (see \cite{gh-2005}). In what follows, we discuss the EOP number of hypercubes.

One of the well-known open problems in hypercubes is concerned with their domination number. Given a graph $G$, a set $D$ is a {\em dominating set} in $G$ if $\bigcup_{v\in D}{N[v]}=V(G)$. The minimum cardinality of a dominating set in $G$ is its {\em domination number}, denoted $\gamma(G)$. The exact value of domination number for $n$-cubes have only been determined when $n\le 9$ (see \cite{ostergard-2001}), and for some special families of these graphs, which we present after some further definitions. 

If $X\subseteq V(G)$, then $X$ is a {\em $2$-packing} in $G$ if $N[x]\cap N[y]=\emptyset$ for every pair of distinct vertices $x,y\in X$. The cardinality of a largest $2$-packing in $G$ is the {\em $2$-packing number} $\rho_2(G)$ of $G$. If $P\subseteq V(G)$ is both a $2$-packing and a dominating set of $G$, then $P$ is a {\em $1$-perfect code} or an {\em efficient dominating set} in $G$. Note that a 1-perfect code in $G$ is precisely a $2$-packing such that the closed neighborhoods of its vertices form a partition of the vertex set of $G$. If $G$ is, in addition, $r$-regular, then we can determine the $2$-packing number and the domination number very easily.

\begin{observation}
\label{obs:code_regular}
If $G$ is an $r$-regular graph admitting a $1$-perfect code, then $\gamma(G)=\rho_2(G)=\frac{|V(G)|}{r+1}$.
\end{observation}

The following result concerning domination number and $2$-packing number of the hypercubes $Q_{2^k-1}$, where $k\in \mathbb{N}$, follows from the fact that they admit a $1$-perfect code.  

\begin{theorem}{\rm(\cite{harary-1993})}\label{thm:infinite-families}
If $n=2^k-1$, where $k\ge 1$, then $\gamma(Q_{n})=\rho_2(Q_n)=2^{n-k}$.
\end{theorem}

Let $G$ be a graph. Note that every 2-packing is a set of vertices in $G$ that are pairwise at distance greater than $2$. We also need a similar concept, where distances between pairs of vertices in a set $P$ have to be greater than $3$ and such a set $P$ is a {\em $3$-packing} in $G$. The largest cardinality of a $3$-packing in $G$ is the {\em $3$-packing number} of $G$ and is denoted by $\rho_3(G)$. If $G$ is bipartite, then this invariant can be useful for bounding the EOP number of $G$.

\begin{lemma}\label{lem:3packingbipartite}
If $G$ is a bipartite graph, then $\roeo(G)\ge \delta(G)\rho_3(G)$.
\end{lemma}
\begin{proof}
Let $G$ be a bipartite graph and let $P$ be a $\rho_3(G)$-set. Since $G$ is bipartite, every $N[u]$ induces the star $K_{1,\deg(u)}$. Letting $S=\bigcup_{u\in P}{E(G[N[u]])}$, and noting that any two edges $e$ and $f$, where $e$ is incident with $u\in P$ and $f$ is incident with $v\in P$, do not have a common edge, we derive that $S$ is an EOP set in $G$. Moreover, $\roeo(G)\ge |S|=\sum_{u\in P}{\deg(u)}\ge \rho_3(G)\delta(G)$.
\end{proof}

In bipartite $r$-regular graphs, the above lower bound simplifies to $\roeo(G)\ge r\rho_3(G)$, which yields the following observation for hypercubes. 

\begin{corollary}\label{cor:cubelower}
If $n\ge 1$, then $\roeo(Q_n)\ge n\rho_3(Q_n)$.
\end{corollary}

Having a potentially useful lower bound on the EOP number of hypercubes, we next focus on the $3$-packing number of hypercubes. First, we prove the following auxiliary result on prisms, which may be of independent interest.

\begin{proposition}\label{prp:3packprism}
If $G$ is a graph, then $\rho_3(G\cp K_2)\le \rho_2(G)$. If, in addition, $G$ is bipartite, then $\rho_3(G\cp K_2)=\rho_2(G)$.
\end{proposition}
\begin{proof}
Let $Q$ be a $\rho_3(G\cp K_2)$-set, and set $V(K_2)=\{1,2\}$. Consider the natural projection $p_G:V(G)\times V(K_{2})\rightarrow V(G)$. Let $u,v\in p_G(Q)$ be two distinct vertices. Since $(u,i)$ and $(v,j)$, where $i,j\in \{1,2\}$, are at distance at least $4$ in $G\cp K_2$, we infer that $4\le d_{G\cp K_2}\big((u,i),(v,j)\big)=d_G(u,v)+d_{K_2}(i,j)\le d_G(u,v)+1$. Thus $d_G(u,v)\ge3$, and so $p_G(Q)$ is a $2$-packing in $G$. Therefore, $\rho_2(G)\ge|p_G(Q)|=|Q|=\rho_3(G\cp K_2)$. 

Now let $G$ be a bipartite graph with partite sets $A_1$ and $A_2$. Let $S$ be a $\rho_2(G)$-set. Consider the set of vertices in $G\cp K_2$ defined as
\begin{center} 
$P=\big{\{}(u,1)\mid u\in S\cap A_1\big{\}}\cup \big{\{}(u,2)\mid u\in S\cap A_2\big{\}}$.
\end{center}
Note that $|P|=|S|=\rho_2(G)$ and we claim that $P$ is a $3$-packing in $G\cp K_2$. Let $(u,i)$ and $(v,j)$ be arbitrary distinct vertices in $P$. Since $S$ is a $2$-packing, $d_G(u,v)\ge 3$. If $d_G(u,v)\ge 4$, then $d_{G\cp K_2}\big((u,i),(v,j)\big)\ge d_G(u,v)\ge 4$, as desired. Hence, assume that $d_G(u,v)=3$. Now we infer that $u$ and $v$ belong to different partite sets of $G$, and so $i\ne j$. Thus, $d_{G\cp K_2}\big((u,i),(v,j)\big)=d_G(u,v)+d_{K_2}(i,j)=3+1=4$, and so $P$ is a $3$-packing, implying $\rho_3(G\cp K_2)\ge |P|=|S|=\rho_2(G)$. Combined with the first statement of the proposition, we get $\rho_3(G\cp K_2)=\rho_2(G)$.
\end{proof}

Open packing and $2$-packing in hypercubes were recently studied in \cite{bkr-2024}, and the exact value for the $2$-packing number of $n$-cubes were obtained for $n\le 7$. Combining this with Proposition \ref{prp:3packprism} and Corollary \ref{cor:cubelower}, we obtain the following table with values of the three involved invariants (namely, the 2-packing number, the $3$-packing number and the EOP number) in hypercubes of small dimensions.

\begin{table}[ht!]
\begin{center}
\begin{tabular}{|l|c|c|c|c|c|c|c|c|}
 \hline
  $n$ & 1 & 2 & 3 & 4 & 5& 6& 7& 8\\
 \hline
  \hline
 $\rho_2=$ & 1 & 1 & 2 & 2& 4 & 8 & 16 & 17-30 \\
 \hline
 $\rho_3=$ & 1 & 1 & 1 & 2 & 2 & 4 & 8 & 16  \\
 \hline
 $\roeo\ge$ & 1 & 2 & 3 & 8 & 10 & 24 & 56 & 128 \\
 \hline
\end{tabular}
\end{center}
\caption{{\small The exact values of the $2$-packing number and the $3$-packing number of $Q_n$ for $n\leq7$ and $n\leq8$, respectively, and lower bounds on the edge open packing number of these $n$-cubes.}}\label{tab:hyper}
\end{table} 

The first row in the table is from \cite{bkr-2024}, the second row follows by Proposition \ref{prp:3packprism}. The third row, which provides lower bounds on $\roeo(Q_n)$, follows by Corollary \ref{cor:cubelower}.
 
By using the earlier discussion on $1$-perfect codes in hypercubes, we are able to obtain the exact values of the EOP number for an infinite family of hypercubes.

\begin{theorem}
\label{thm:infinite-families-roeo}
If $n=2^k$, where $k\ge 1$, then $\roeo(Q_{n})=2^{n-1}$.
\end{theorem}
\begin{proof}
By Theorem \ref{thm:infinite-families}, we have  $\rho_2(Q_{2^k-1})=2^{2^k-1-k}$, where $k\ge 1$. Combining this with Proposition \ref{prp:3packprism}, we infer $$\rho_3(Q_{2^k})=\rho_3(Q_{2^k-1}\cp K_2)=\rho_2(Q_{2^k-1})=2^{2^k-1-k}.$$ Now, using Corollary \ref{cor:cubelower}, we get $$\roeo(Q_{2^k})\ge 2^k\rho_3(Q_{2^k})=2^k2^{2^k-1-k}=2^{2^k-1}.$$
Finally, combining this with $\alpha(Q_{2^k})=2^{2^k-1}$, and the general bound $\roeo(G)\le \alpha(G)$ which holds for any graph $G$, we derive the statement of the theorem. 
\end{proof}

It may be interesting to remark that similarly as the EOP number, the open packing number of hypercubes have also been determined for the dimensions $n\le 8$ and the dimensions $n$, which are powers of $2$ (see \cite{bkr-2024}). In particular, $\rho^o(Q_{2^k})=2^{2^k-k}$. Again, when $n=9$, the exact value is not known, we only have the bounds $34\le \rho^o(Q_9)\le 60$. Similarly, using the lower bound for $\rho_2(Q_8)$ from Table \ref{tab:hyper} and the earlier reasoning, we infer $\roeo(Q_9)\ge 9\times17=153$. 


\section{Rooted product}
\label{sec:rooted}

A \emph{rooted graph} is a graph in which one vertex is labeled in a special way to distinguish it from the other vertices. This vertex is called the \emph{root} of the graph. Let $G$ be a graph with vertex set $\{v_1,\ldots,v_n\}$. Let ${\cal H}$ be a sequence of $n$ rooted graphs $H_1,\ldots,H_n$. The \emph{rooted product graph} $G({\cal H})$ is the graph obtained by identifying the root of $H_i$ with $v_{i}$ (see \cite{GM}). We here consider the particular case of rooted product graphs where ${\cal H}$ consists of $n$  isomorphic rooted graphs \cite{Sch}. More formally, assuming that the root of $H$ is $v$, we define the rooted product graph $G\circ_{v}H=(V,E)$, where $V=V(G)\times V(H)$ and
\begin{center}
$E=\bigcup_{i=1}^{n}\Big{\{}(v_i,h)(v_i,h')\mid hh'\in E(H)\Big{\}}\bigcup \Big{\{}(v_i,v)(v_j,v)\mid v_iv_j\in E(G)\Big{\}}.$
\end{center}

Note that the subgraphs induced by $H$-fibers of $G\circ_v H$ are isomorphic to $H$. We next give a closed formula for the induced matching number of rooted product graphs. (Recall that $\beta(G)$ stands for the vertex cover number of $G$.)

\begin{theorem}\label{root}
Let $G$ be any graph of order $n$. If $H$ is any graph with root $v$, then
$$\nu_{I}(G\circ_{v}H)\in \Big{\{}n\nu_{I}(H)-\beta(G),n\nu_{I}(H),n\nu_{I}(H)+\nu_{I}(G)\Big{\}}.$$
\end{theorem}
\begin{proof}
For the sake of convenience, we make use of the notations $G_{v}=(G\circ_{v}H)[V(G)\times \{v\}]$ and $H_{i}=(G\circ_{v}H)[\{v_{i}\}\times V(H)]$ for each $i\in[n]$ throughout the proof. Also, given a graph $F$, let $E_{F}(x)$ stand for the set of edges of $F$ incident with $x\in V(F)$. We observe that $G_{v}\cong G$ and $H_{i}\cong H$ for each $i\in[n]$.

We first claim that 
\begin{equation}\label{Equ1}
n\nu_{I}(H)-\beta(G)\leq \nu_{I}(G\circ_{v}H)\leq n\nu_{I}(H)+\nu_{I}(G).
\end{equation}
To prove the claim, let $B$ be a $\nu_{I}(G\circ_{v}H)$-set. Note that $B\cap E(G_{v})$ and $B\cap E(H_{i})$ are induced matchings in $G_{v}$ and $H_{i}$, for each $i\in[n]$, respectively. Therefore,
$$\nu_{I}(G\circ_{v}H)=|B|=|B\cap E(G_{v})|+\sum_{i=1}^{n}|B\cap E(H_{i})|\leq \nu_{I}(G)+n\nu_{I}(H).$$

In order to prove the lower bound in (\ref{Equ1}), let $Q$ be a $\nu_{I}(H)$-set. Clearly, $Q_{i}=\{(v_{i},h)(v_{i},h')\mid hh'\in Q\}$ is a $\nu_{I}(H_{i})$-set for every $i\in[n]$. Since $Q_{i}$ is an induced matching in $H_{i}$, it follows that $|E_{H_{i}}\big{(}(v_{i},v)\big{)}\cap Q_{i}|\leq1$ for all $i\in[n]$. If $I$ is an $\alpha(G)$-set, then 
$$Q'=(\bigcup_{v_{i}\in I}Q_{i})\bigcup\Big{(}\bigcup_{v_{i}\notin I}\big{(}Q_{i}\setminus(E_{H_{i}}\big{(}(v_{i},v)\big{)}\cap Q_{i})\big{)}\Big{)}$$   
is an induced matching in $G\circ_{v}H$. Together this fact and the Gallai theorem (which states that $\alpha(F)+\beta(F)=|V(F)|$ for each graph $F$) imply that 
$$\nu_{I}(G\circ_{v}H)\geq|Q'|\geq \alpha(G)\nu_{I}(H)+\beta(G)(\nu_{I}(H)-1)=n\nu_{I}(H)-\beta(G).$$
In fact, we have proved both inequalities in (\ref{Equ1}). 

We now consider the following cases.\vspace{0.75mm}\\
\textit{Case 1}. There exists a $\nu_{I}(H)$-set $A$ for which $N_{H}(v)\cap V(H[A])=\emptyset$ \big{(}this in particular means that $v\notin V(H[A])$\big{)}. It is then easy to see that 
\begin{center}
$\{(v_{i},h)(v_{i},h')\mid hh'\in A\ \mbox{and}\ i\in[n]\}\cup \{(g,v)(g',v)\mid gg'\in S\}$
\end{center} 
is an induced matching in $G\circ_{v}H$ of cardinality $n\nu_{I}(H)+\nu_{I}(G)$, in which $S$ is a $\nu_{I}(G)$-set. Hence, $\nu_{I}(G\circ_{v}H)\geq n\nu_{I}(H)+\nu_{I}(G)$, leading to equality in the upper bound in (\ref{Equ1}).\vspace{0.75mm}\\
\textit{Case 2}. $N_{H}(v)\cap V(H[A])\neq\emptyset$ for each $\nu_{I}(H)$-set $A$. Let $k=|B\cap E(G_{v})|$ (note that $k$ may be equal to zero). Since $B$ is also an induced matching in $G_{v}$, precisely $2k$ vertices of $G_{v}$ are incident with the edges in $B\cap E(G_{v})$. Let $J$ be the set of all these $2k$ vertices. In view of this and since $N_{H}(v)\cap V(H[A])\neq\emptyset$ for each $\nu_{I}(H)$-set $A$, we deduce that $|B\cap E(H_{i})|\leq \nu_{I}(H)-1$ when $v_{i}\in J$. Therefore, we get
\begin{equation}\label{Roo}
\begin{array}{lcl}
\nu_{I}(G\circ_{v}H)=|B|&=&|B\cap E(G_{v})|+\sum_{i\in J}|B\cap E(H_{i})|+\sum_{i\notin J}|B\cap E(H_{i})|\\
&\leq& k+2k(\nu_{I}(H)-1)+(n-2k)\nu_{I}(H)=n\nu_{I}(H)-k\leq n\nu_{I}(H).
\end{array}
\end{equation}

Next, we need two distinguish two more possibilities depending on the behavior of $\nu_{I}(H)$-sets.\vspace{0.75mm}\\
\textit{Subcase 2.1.} $v\in V(H[A])$ for each $\nu_{I}(H)$-set $A$. Let $B'$ be the subset of those edges from $B$ which are incident with a vertex in $G_{v}$. We set $B'=B'_{1}\cup B'_{2}$, in which every edge of $B'_{2}$ is incident with two vertices of $G_{v}$ and $B'_{1}=B\setminus B'_{2}$. Note that every edge in $B'_{1}$ is in $H_{i}$ for some $i\in[n]$. Because $B$ is also an induced matching in $G_{v}$, it follows that $|B'|\leq \alpha(G)$. Moreover, we infer from the assumption of this subcase that $|B\cap E(H_{i})|\leq \nu_{I}(H)-1$ when $v_{i}$ is incident with an edge in $B'_{2}$. Altogether, we get
\begin{equation*}\label{Roo2}
\begin{array}{lcl}
\nu_{I}(G\circ_{v}H)&=&|B|\leq|B'_{2}|+2|B'_{2}|(\nu_{I}(H)-1)+|B'_{1}|\nu_{I}(H)+(n-2|B'_{2}|-|B'_{1}|)(\nu_{I}(H)-1)\\
&=&n(\nu_{I}(H)-1)+|B'|\leq n(\nu_{I}(H)-1)+\alpha(G)=n\nu_{I}(H)-\beta(G).
\end{array}
\end{equation*}
This leads to $\nu_{I}(G\circ_{v}H)=n\nu_{I}(H)-\beta(G)$ due to the first inequality in (\ref{Equ1}).\vspace{0.75mm}\\
\textit{Subcase 2.2.} $v\notin V(H[A])$ for some $\nu_{I}(H)$-set $A$. In such a situation, $\bigcup_{i=1}^{n}\{(v_{i},h)(v_{i},h')\mid hh'\in A\}$ is an induced matching in $G\circ_{v}H$ of cardinality $n\nu_{I}(H)$. So, $\nu_{I}(G\circ_{v}H)\geq n\nu_{I}(H)$. We now infer $\nu_{I}(G\circ_{v}H)=n\nu_{I}(H)$ from (\ref{Roo}).\vspace{0.75mm}

All in all, we have proved that $\nu_{I}(G\circ_{v}H)$ belongs to $\big{\{}n\nu_{I}(H)-\beta(G),n\nu_{I}(H),n\nu_{I}(H)+\nu_{I}(G)\big{\}}$.   
\end{proof}

In what follows, we show that $\nu_{I}(G\circ_{v}H)$ can assume any of the three values in the statement of Theorem \ref{root}. Let $r \geq 2$ and let $H=S(n_1,\ldots,n_r)$. First let $n_1=2$ and $n_i=1$ for all $i\in \{2,\ldots,r\}$ and let the root of $H$ be the vertex at distance 2 from the center of $H$. It is easy to observe that for every graph $G$ of order $n$, we have $\nu_{I}(G\circ_{v}H)=n+\nu_{I}(G)=n\nu_{I}(H)+\nu_{I}(G)$.

Now let $n_i=1$ for all $i \in [r]$ ($H$ is a star) and let the root $v$ be a leaf of the star. Then, $\nu_{I}(G\circ_{v}H)=n=n\nu_{I}(H)$ for all graphs $G$.

Finally, let $n_1=n_2=2$, $n_i=1$ for all $i \in \{3,\ldots ,r\}$ and let $v$ be a leaf of $H$ that is at distance 2 from the center of $H$. Then, $\nu_{I}(G\circ_{v}H)=n+\alpha(G)=2n-\beta(G)=n\nu_{I}(H)-\beta(G)$ for all graphs $G$.

Let $G$ and $H$ be graphs where $V(G)=\{v_1,\ldots,v_{n}\}$. The \textit{corona product} $G\odot H$ of the graphs $G$ and $H$ is obtained from the disjoint union of $G$ and $n$ disjoint copies of $H$, say $H_1,\ldots,H_{n}$, such that $v_i\in V(G)$ is adjacent to all vertices of $H_i$ for each $i\in[n]$. As a consequence of Theorem \ref{root} and its proof, we obtain the exact formula of the induced matching number of corona product graphs. Recall first that the \textit{join} of graphs $G$ and $H$, written $G\vee H$, is a graph obtained from the disjoint union $G$ and $H$ by adding the edges $\{gh\mid g\in V(G)\ \mbox{and}\ h\in V(H)\}$. 

\begin{corollary}\label{Cor} 
For any graphs $G$ and $H$, 
\begin{equation*}\label{EQ1}
\nu_{I}(G\odot H)=\left \{
\begin{array}{lll}
|V(G)|\nu_{I}(H) & \mbox{if}\ E(H)\neq \emptyset,\vspace{1.5mm}\\
\alpha(G) & \mbox{if}\ E(H)=\emptyset.
\end{array}
\right.
\end{equation*}
\end{corollary}
\begin{proof}
We observe that $G\odot H$ is isomorphic to $G\circ_{v}(K_{1}\vee H)$, in which the root $v$ is the unique vertex of $K_{1}$. In view of this and Theorem \ref{root}, we have 
$$\nu_{I}(G\odot H)\in \big{\{}n\nu_{I}(K_{1}\vee H)-\beta(G),n\nu_{I}(K_{1}\vee H),n\nu_{I}(K_{1}\vee H)+\nu_{I}(G)\big{\}}.$$ 

Suppose first that $H$ is not edgeless. It is then a routine matter to see that $\nu_{I}(K_{1}\vee H)=\nu_{I}(H)$. In view of this, there exists a $\nu_{I}(K_{1}\odot H)$-set $A$ containing no edges in $\{vh\mid h\in V(H)\}$. So, $A$ fulfills the statements of Case $2$ and Subcase $2.2$ in the proof of Theorem \ref{root}. Therefore, $\nu_{I}(G\odot H)=|V(G)|\nu_{I}(H)$.

Now let $H$ be edgeless. It is then clear that $\nu_{I}(K_{1}\vee H)=1$. Therefore, $\nu_{I}(G\odot H)\in \big{\{}n-\beta(G),n,n+\nu_{I}(G)\big{\}}$ by Theorem \ref{root}. On the other hand, we observe that any $\nu_{I}(K_{1}\vee H)$-set satisfies the statements of Case $2$ and Subcase $2.1$ in the proof of Theorem \ref{root}. Therefore, $\nu_{I}(G\odot H)=n-\beta(G)=\alpha(G)$.
\end{proof}

An analogue to the inequality (\ref{Equ1}) concerning the EOP number can be given as follows: 
\begin{equation}\label{Equ2}
n\eop(H)-\deg_{H}(v)\beta(G)\leq \eop(G\circ_{v}H)\leq n\eop(H)+\eop(G).
\end{equation}

Note that both the lower and the upper bounds in (\ref{Equ2}) are sharp. For the lower bound, it suffices to consider the graph $C_{4}\circ_{v}K_{1,r}$ for $r\geq2$, in which $v$ is the center of $K_{1,r}$. It is easy to see that for even integer $n\geq4$, $\eop(C_{n}\circ_{v}K_{1,r})=\frac{n}{2}r=n\eop(K_{1,r})-\deg_{K_{1,r}}(v)\beta(C_{n})$ holds. To see the sharpness of the upper bound, let $G$ be any graph and $H=S(3,1,\ldots,1)$. Let $v$ be the vertex in $H$ at distance $3$ from its center. Then, $\eop(G\circ_{v}H)=nr+\eop(G)=n\eop(H)+\eop(G)$.

Several open problems arise from this paper, yet we expose the following one for which we suspect is resolvable. 
\begin{problem}
Can one determine a closed formula for $\eop(G\circ_{v}H)$?
\end{problem}

\section*{Acknowledgements}
The authors were supported in part by the Slovenian Research and Innovation Agency (ARIS) (grants P1-0297, N1-0285, J1-3002, and J1-4008).


\end{document}